\documentclass{amsart}

\usepackage{amscd,amsfonts,amsmath,amssymb,amsthm}
\usepackage{epsf,verbatim,xy}
\usepackage[mathscr]{eucal}
\usepackage[dvips]{graphicx}
\usepackage[latin5]{inputenc}
\usepackage{bbm,bm}
\usepackage{stmaryrd}
\usepackage{color}
\usepackage{latexsym}
\usepackage{yfonts}
\usepackage{mathrsfs}
\usepackage{mathtools}

\allowdisplaybreaks

\newtheorem{Thm}{Theorem}
\newtheorem{Lem}{Lemma}
\newtheorem{Cor}{Corollary}

\theoremstyle{remark}
\theoremstyle{remark}\newtheorem{Rem}{Remark}
\theoremstyle{remark}
\numberwithin{equation}{section}

\def\sumprime_#1{\setbox0=\hbox{$\scriptstyle{#1}$}
\setbox2=\hbox{$\displaystyle{\sum}$}
\setbox4=\hbox{${}'\mathsurround=0pt$}
\dimen0=.5\wd0 \advance\dimen0 by-.5\wd2
\ifdim\dimen0>0pt
\ifdim\dimen0>\wd4 \kern\wd4 \else\kern\dimen0\fi\fi
\mathop{{\sum}'}_{\kern-\wd4 #1}}

\newcommand{\Z}{{\mathbb Z}}

\newcommand\norm[1]{\left\lVert#1\right\rVert}

\begin{document}

\title{Asymptotic evaluation of a
lattice sum associated with the Laplacian matrix
}

\author{Arzu Boysal \hspace{.3cm} Fatih Ecevit \hspace{.3cm} Cem Yal\c{c}{\i}n Y{\i}ld{\i}r{\i}m} 
\address{Department of Mathematics, Bo\~{g}az{\tiny \.{I}}\c{c}{\tiny \.{I}}  University, \.{I}stanbul 34342, Turkey}
\email{arzu.boysal@boun.edu.tr, fatih.ecevit@boun.edu.tr, yalciny@boun.edu.tr} 
\date{\today}


\begin{abstract}
The Laplacian matrix is of fundamental importance in the study of graphs, networks, random walks on lattices, and arithmetic of curves. In certain cases, the trace of its pseudoinverse appears as the only non-trivial term in computing some of the intrinsic graph invariants. Here we study a double sum $F_n$ which is associated with the trace of the pseudo inverse of the Laplacian matrix for certain graphs. We investigate the asymptotic behavior of this sum as $n \to \infty$. Our approach is based on classical analysis combined with asymptotic and numerical analysis, and utilizes special functions. We determine the leading order term, which is of size $n^2 \, \log n$, and develop general methods to obtain the
secondary main terms in the asymptotic expansion of $F_n$ up to errors of $\mathcal{O}(\log n)$ and $\mathcal{O}(1)$ as $n \to \infty$. We provide some examples to demonstrate our methods.
\end{abstract}
\maketitle \markright{Asymptotic behavior of planar lattice sums}

\section{introduction} \label{sec:intro}

The Laplacian matrix $\mathcal{L}$, defined as the difference of the degree and adjacency matrices for graphs (or networks) \cite[\S2.2]{Poz}, plays an important role in a variety of fields encompassing graph theory, random walks on lattices, analysis of grids, and arithmetic of curves (see e.g. \cite{Hu,Cj,Poz,Ye,Gu,BGPWZ,C1,C2} and the references therein). Of fundamental interest is the trace
$\operatorname{tr}(\mathcal{L}^+)$, where $\mathcal{L}^+$ is the pseudoinverse of the Laplacian matrix $\mathcal{L}$. Under doubly periodic boundary conditions \cite{Poz,Ye}, 
%{\color{red}(see e.g. \cite[3.1.3, 3.3.2, 3.4.2, 3.5.3]{Poz})} 
this trace is related to the sum
\begin{equation} \label{eq:Fn-concreteT}
	F_n= \!\!\! \sum_{j,k=0 \atop (j,k)\neq (0,0)}^{n-1} {1\over 1- {1\over L}\sum\limits_{\ell = 1}^{L} \cos (s_{1,\ell} {2\pi j \over n} + s_{2,\ell} {2\pi k \over n})}
\end{equation}
where $L \ge 2$ is an integer, $s_{1,\ell},s_{2,\ell}$ are non-zero integers with $(s_{1,1},s_{2,1}) = (1,0)$ and $(s_{1,2},s_{2,2}) = (0,1)$.
%\begin{equation} \label{eq:Fn-concrete}
%	F_n= \!\!\! \sum_{j,k=0 \atop (j,k)\neq (0,0)}^{n-1} {1\over 1- {1\over L}\sum\limits_{\ell = 1}^{L} \cos (\mathbf{s}_{\ell} \cdot \mathbf{t}_{j,k})}
%\end{equation}
%where $L \ge 2$ is a natural number, $\mathbf{s}_{\ell} = (s_{1,\ell},s_{2,\ell}) \in \mathbb{Z}^2 \backslash \{ \mathbf{0} \}$ for
%$\ell = 1,\ldots,L$ with $\mathbf{s}_{1} = (1,0)$ and $\mathbf{s}_{2} = (0,1)$, and $\boldsymbol{t}_{j,k} = (t_j,t_k) = ({2\pi j \over n}, {2\pi k \over n })$
%for $j,k \in \mathbb{Z}$.
For example, when the graph is related to the square lattice
\[
	\operatorname{tr}(\mathcal{L}^+)
	= {1 \over 4} \!\!\! \sum_{j,k=0 \atop (j,k)\neq (0,0)}^{n-1} {1\over 1- {1\over 2} ( \cos{2\pi j \over n} + \cos{2\pi k \over n})},
\]
and when it corresponds to the triangular lattice
\[
	\operatorname{tr}(\mathcal{L}^+)
	= {1 \over 6} \!\!\! \sum_{j,k=0 \atop (j,k)\neq (0,0)}^{n-1} {1\over 1- {1\over 3} ( \cos{2\pi j \over n}+ \cos{2\pi k \over n} + \cos{2\pi (j+k) \over n} )}.
\]

Determining the exact explicit dependence of $F_n$ on $n$, and thus that of $\operatorname{tr}(\mathcal{L}^+)$, seems to be a nontrivial problem, and there are no results in the literature in this direction. Here we study the asymptotic behavior of $F_n$ as $n\to\infty$. The analysis of this asymptotic behavior directly unravels the asymptotic behavior of  $\operatorname{tr}(\mathcal{L}^+)$ as well as the two intrinsic invariants associated with a graph $G$, namely the tau constant $\tau(G)$ and the Kirchhoff index $\operatorname{Kf}(G)$ \cite{BR,Rumely,Ye}. Indeed, for the class of $\mathsf{d}$-regular, equi-resistant, normalized metrized graphs with equal edge lengths, without multiple edges, bridges and self loops (for definitions see e.g. \cite{BeinekeWilson,Rumely}) one has \cite{C1,C2}
\[
	\tau(G)
	= {1 \over 12} \Big( 1 - {2 (\nu-1)\over \mathsf{d} \, \nu} \Big)^2 + {1 \over \nu} \operatorname{tr}(\mathcal{L}^+)
	={1 \over 12} \Big( 1 - {2 (\nu-1) \over \mathsf{d} \, \nu} \Big)^2 + {1 \over \nu^2} \operatorname{Kf}(G)  
\]
where $\nu$ is the number of vertices of $G$. This class contains graphs associated with certain complex semi-simple Lie algebras of rank $2$, e.g. those of $\mathfrak{sl}_2\times \mathfrak{sl}_2$ and $\mathfrak{sl}_3$ (see \cite{Bo} for details) whose root lattices are otherwise known as the square lattice and the triangular lattice respectively.

The approach we develop in this paper for understanding the asymptotic behavior of $F_n$ is a generalization of our work on
the triangular lattice \cite{BEY} where we studied a related sum
\[
	S_n = \sum_{j,k=1}^{n-1} \frac{1}{3-\cos({2\pi j \over n})-\cos({2\pi k \over n})-\cos({2\pi (j+k) \over n})}.
\]
Using the relation $S_n = {1\over 3} ( F_n - {1 \over 2} (n^2-1))$, the main result in \cite{BEY} is
\[
	F_n = {\sqrt{3} \over \pi} n^2\log n + {\sqrt{3} \over \pi}\Big( \gamma + \log(4\pi\sqrt[4]{3}) - 3\log\Gamma\big({1\over 3}\big) \Big) n^2 + \mathcal{O}(\log n),
	\quad
	\text{as } n \to \infty,
\]
%\[
%	S_n = {n^2\log n \over \sqrt{3}\pi}+ {n^2\over \sqrt{3}\pi}\Big( \gamma -{\sqrt{3}\pi\over 6}+\log(4\pi\sqrt[4]{3}) -3\log\Gamma\big({1\over 3}\big) \Big) +O\big(\log n\big)
%\]
where $\gamma$ is Euler's constant. To study the asymptotic behavior of $S_n$, we utilized the first non-zero term in the Taylor approximation of the denominator
of the summand. This term is a positive definite binary quadratic form, and we used it to set up various sums and integrals approximating the sum $S_n$. We showed that
a combination of these sums and integrals facilitates the determination of the asymptotic behavior of $S_n$ within an error term of magnitude $\mathcal{O}(\log n)$.
Here we show that the first non-zero term in the Taylor approximation of the denominator of the summand in the general setting of the sum $F_n$ in \eqref{eq:Fn-concreteT}
is still a positive definite binary quadratic form, and the same combination of the related sums and integrals allows for the determination of the asymptotic
behavior of $F_n$ with an error of $\mathcal{O}(\log n)$. We also prove that the utilization of any higher order Taylor approximation of the denominator in the summand implies
an error of $\mathcal{O}(1)$.

The paper is organized as follows. In \S2 we present an abstract coordinate free formulation of $F_n$ in \eqref{eq:Fn-concreteT} which allows us to deduce that it is rational for any $n$. We also derive a more suitable representation of $F_n$ for studying its asymptotic behavior. In \S\ref{sec:mainresults}, we
relate the sum $F_n$ to sums and integrals obtained by Taylor approximations of the denominator of the summands. The first main result therein, Theorem~\ref{thm:main1}, shows
that these sums and integrals all share the same leading term (of magnitude $n^2 \log n$) in their asymptotic expansions with that of $F_n$.
For obtaining the next two terms (of sizes $n^2$ and $n$) in the asymptotic expansion of $F_n$, we establish in Theorem~\ref{thm:main2} a relation among
these sums and integrals, and we obtain error terms of $\mathcal{O}(\log n)$ or $\mathcal {O}(1)$ depending on the degree of the Taylor polynomial utilized. For understanding the implication of the use of the first non-zero term in the Taylor approximation, we investigate the asymptotic behavior of the sum of reciprocals of a positive definite binary quadratic form at the lattice points and state our result in Theorem~\ref{thm:Gn}. We then use this to derive the asymptotic behavior of the associated sum in Theorem~\ref{thm:Fnf1}. The asymptotic behavior of a related integral is stated as Theorem~\ref{thm:Inf1}.
The content of \S4 is some lemmata and the proofs of Theorems~\ref{thm:main1}, \ref{thm:main2} and \ref{thm:Inf1}.
The proofs of Theorems~\ref{thm:Gn} and \ref{thm:Fnf1} are given in \S\ref{sec:Gn} and \S\ref{sec:Fnf1} respectively.
In \S\ref{sec:square}, we apply Theorem~\ref{thm:main2} to two examples coming from Lie theory.
The first example corresponds to the root lattice of $\mathfrak{sl}_2\times \mathfrak{sl}_2$ (the square lattice) whereas the second example is for the
root lattice of $\mathfrak{so}_5$ (the union jack lattice).
%The same analysis could be carried out for the remaining
%rank 2 case, $G_2$, which requires pedestrian calculations not included here.
Some problems and directions for further research are noted in \S\ref{sec:fp}.

{\it Notation.}
Vectors are written boldface. We take an empty sum to be zero. The letter $\gamma$ denotes Euler's constant.

\section{Abstract and alternative formulations of $F_n$}

Let $\Phi$ be a finite collection of nonzero vectors in a two dimensional real vector space  generating a lattice $\Lambda =\Z\Phi$ of full rank  with basis in $\Phi$.
We remark that the conditions on $\Phi$ are satisfied for positive roots of complex semisimple Lie algebras of rank 2 (they are of type $A_1\times A_1$, $A_2$, $B_2$ or $G_2$,
see \cite{Bo} for details).

We denote the number of elements of $\Phi$  by $|\Phi|$, and define a function $\phi$ associated to $\Phi$ as
\[
	\phi(x) := \frac{1}{2|\Phi|} \sum_{\mathbf{s} \in \pm \Phi} e^{ i \langle \mathbf{s}, x \rangle},	
\]
where the bracket $\langle \cdot, \cdot \rangle$ denotes the standard duality pairing. 

Let us write $\Phi = \{ \mathbf{b}_1, \ldots, \mathbf{b}_{|\Phi|} \}$ and assume that $\{ \mathbf{b}_1, \mathbf{b}_2\}$
is a basis for $\Lambda$, and denote the associated dual basis by $\{\ss_1,\ss_2 \}$, that is $\langle \mathbf{b}_i,\ss_j\rangle=\delta_{ij}$.  
For a given positive integer $n$, let 
\[
	t_j := \frac{2\pi j}{n}
	\quad (j \in \mathbb{Z}), \;\; \text{and}\;\;
	x_{j,k} := t_j \, \ss_1 + t_k \, \ss_2
	\quad
	(j,k \in \mathbb{Z}),
\]
and consider the sum
\begin{equation}
	F_n(\Phi) := \sum_{j,k=0 \atop (j,k)\neq (0,0)}^{n-1} \frac{1}{1-\phi(x_{j,k})}. \label{eq:sum}
\end{equation}
Since $\phi$ has rational coefficients and the summation is invariant under the action of the Galois  group of the cyclotomic field obtained by adjoining a primitive $n$-th root of unity to $\mathbb Q$, $F_n(\Phi)$ takes rational values for any $n$.

The sum $F_n(\Phi)$ depends only on $\Phi$ and $n$, and not on our choice of a basis of $\Lambda$ from $\Phi$.
Indeed if $\{ \mathbf{b}_1,\mathbf{b}_2 \} \subset \Phi$ is a basis for $\Lambda$, we have $\mathbf{b}_{\ell} = s_{1,\ell} \mathbf{b}_1 + s_{2,\ell} \mathbf{b}_2$ $(\ell =1,\ldots,|\Phi|)$ for some $s_{1,\ell},s_{2,\ell} \in \mathbb{Z}$,
and therefore
\begin{equation} \label{eq:key}
	\langle \mathbf{b}_{\ell}, x_{j,k} \rangle
	= \langle s_{1,\ell} \mathbf{b}_1 + s_{2,\ell} \mathbf{b}_2, t_j \, \ss_1 + t_k \, \ss_2 \rangle
	= s_{1,\ell} t_j + s_{2,\ell} t_k
	= \mathbf{s}_{\ell} \cdot \boldsymbol{t}_{j,k},
\end{equation}
where $\boldsymbol{t}_{j,k} := (t_j,t_k)$ for $j,k \in \mathbb{Z}$, $\cdot$ is the standard Euclidean inner product in $\mathbb{R}^2$, and
$\mathbf{s}_{\ell} = (s_{1,\ell},s_{2,\ell})$ for $\ell = 1,\ldots, |\Phi|$ so that
\begin{equation} \label{eq:s1-s2}
	\mathbf{s}_{1} = (1,0)
	\quad
	\text{and}
	\quad
	\mathbf{s}_{2} = (0,1).
\end{equation}
If $\{ \mathbf{b}_1',\mathbf{b}_2' \} \subset \Phi$ is another basis for $\Lambda$, then we formally have
\[
	\begin{bmatrix}
		\alpha_{11} & \alpha_{12} \\
		\alpha_{21} & \alpha_{22}
	\end{bmatrix}
	\begin{bmatrix}
		\mathbf{b}_1 \\
		\mathbf{b}_2
	\end{bmatrix}
	= \begin{bmatrix}
		\mathbf{b}_1' \\
		\mathbf{b}_2'
	\end{bmatrix}
\]
for some unimodular matrix  $\alpha = (\alpha_{ij})$. The associated dual bases are then related by
\[
	(\alpha^{-1})^T
	\begin{bmatrix}
		\ss_1 \\
		\ss_2
	\end{bmatrix}
	= \begin{bmatrix}
		\ss_1' \\
		\ss_2'
	\end{bmatrix},
\]
and, for $y_{j,k} = t_j \ss_1' + t_k \ss_2'$, we have
\[
	\langle \mathbf{b}_{\ell}, y_{j,k} \rangle
	= \begin{bmatrix}
		s_{1,\ell} & 
		s_{2,\ell}
	\end{bmatrix} 
	\alpha^{-1}
	\begin{bmatrix}
		t_j \\
		t_k
	\end{bmatrix}
	= {2\pi \over n} \
	\mathbf{s}_{\ell}
	\cdot \alpha^{-1}
	\begin{bmatrix}
		j \\
		k
	\end{bmatrix}.
\]
Since $\mathbb{Z}_n\times \mathbb{Z}_n$ is invariant under unimodular transformations, it follows that the sum $F_n(\Phi)$
is independent of the choice of the basis.

Using \eqref{eq:key}, we see that the sum $F_n(\Phi)$ can be expressed as
\begin{equation} \label{eq:Fn-concrete}
	F_n (\Phi)
	= F_n(f)
	:= \!\!\!\! \sum_{j,k=0 \atop (j,k)\neq (0,0)}^{n-1} \!\! f(\mathbf{t}_{j,k})
	= \!\!\!\!\sum_{\mathbf{t}_{j,k} \in [0,2\pi)^2 \backslash \{ \mathbf{0} \}} \!\!\!\!  f(\mathbf{t}_{j,k})
	\quad
	\text{with } \
	f(\mathbf{x}) := \frac{1}{\psi(\mathbf{x})}
\end{equation}
where
\begin{equation} \label{eq:psi}
	\psi(\mathbf{x})
	:= 1 - \frac{1}{|\Phi|} \sum_{\ell = 1}^{|\Phi|} \cos (\mathbf{s}_{\ell} \cdot \mathbf{x})
	= \frac{2}{|\Phi|} \sum\limits_{\ell=1}^{|\Phi|} \sin^2(\frac{1}{2} \mathbf{s}_{\ell} \cdot \mathbf{x}),
	\qquad
	\mathbf{x} = (x_1,x_2) \in \mathbb{R}^2.
\end{equation}
Denoting $|\Phi|$ by $L$, we see that $F_n$ in \eqref{eq:Fn-concreteT} is a representation of $F_n(\Phi)$.

The function $f$ is singular at the corners of the rectangle $[0,2\pi] \times [0,2\pi]$. Separating the sum in
\eqref{eq:Fn-concrete} into two parts consisting of $\mathbf{t}_{j,k} \in [0,\pi) \times [0,2\pi)$ and
$\mathbf{t}_{j,k} \in [\pi,2\pi) \times [0,2\pi)$,
we note that the latter may be brought into $\mathbf{t}_{j,k} \in [-\pi,0) \times [0,2\pi)$ without altering the sum. This is because
$s_{1,\ell} \in \mathbb{Z}$ and therefore $\psi$ is $2\pi$-periodic in the first argument of $\mathbf{x} = (x_1,x_2)$. Accordingly,
the sum can be evaluated over $\mathbf{t}_{j,k} \in [-\pi,\pi)\times [0,2\pi) \backslash \{ \mathbf{0} \}$. By a similar argument, using the
$2\pi$-periodicity of $\psi$ in the second argument of $\mathbf{x} = (x_1,x_2)$, we obtain
\begin{equation} \label{eq:Fn-concrete1}
	F_n (f) = \sum_{\mathbf{t}_{j,k} \in [-\pi,\pi)^2 \backslash \{ \mathbf{0} \}} f(\mathbf{t}_{j,k}).
\end{equation}
Aside from a singularity at the origin, the function $f$ is
smooth in the rectangle $[-\pi,\pi] \times [-\pi,\pi]$. To simplify the notation in what follows, we set
\begin{equation} \label{eq:D_n}
	D_n
	:= \bigcup_{\mathbf{t}_{j,k} \in [-\pi,\pi)^2 \backslash \{ \mathbf{0} \}} X_{j,k}
	= 	\left\{
		\begin{array}{ll}
			[-\pi,\pi]^2 \backslash [-\frac{\pi}{n},\frac{\pi}{n}]^2, & n \text{ odd},
			 \vspace{0.1cm} \\
			\left[ -\pi-\frac{\pi}{n},\pi-\frac{\pi}{n} \right]^2 \backslash \left[ -\frac{\pi}{n},\frac{\pi}{n} \right]^2, & n \text{ even},
		\end{array}
	\right.	
\end{equation}
where $X_{j,k} = X_j \times X_k$ with $X_j = [x_j,x_{j+1}]$ and $x_j = t_j - \frac{\pi}{n}$ for $j \in \mathbb{Z}$, and we note that the sum
$F_n(f)$ can also be expressed as
\[
	F_n (f) = \sum_{\mathbf{t}_{j,k} \in D_n} f(\mathbf{t}_{j,k}).
\]

\section{Statement of the main results} \label{sec:mainresults}

In this section, we introduce the approximations of $F_n(f)$ in the form of integrals and sums, and state our main results.

Let us first note that the sum $F_n(f)$ corresponds exactly to the approximation of the integral
\begin{equation} \label{eq:Inf}
	I_{n}(f) := \dfrac{1}{\Delta_n^2} \iint_{D_n} f(\mathbf{x}) \, d\mathbf{x},
	\qquad \big( \Delta_n := \frac{2\pi}{n} \big),
\end{equation}
by applying the product cubature rule, with both factors coming from the midpoint rule, on each of the rectangles $X_{j,k}$ in $D_n$.
As we shall see below, $I_{n}(f)$ and $F_n(f)$ share the same leading order term in their asymptotic expansions, but this is not true
for secondary terms. Accordingly, the derivation of lower order terms in the asymptotic expansion of $F_n(f)$ calls for a different
approach. To this end, as in \cite{BEY}, we resort to the Taylor series expansion of the function $\psi$
appearing in the denominator of $f$. Unlike \cite{BEY}, however, we consider approximation of $\psi$ not
just by the first non-zero term in its Taylor series expansion but also terms of arbitrarily high order. As we show in
Lemma~\ref{Lemma:psi-and-derivatives} below, for $m \ge 1$, the $2m$-th order Taylor polynomial approximation of $\psi$ around
the origin is given by
\[
	p_{m}(\mathbf{x})
	: = \frac{1}{|\Phi|} \sum_{\ell = 1}^{|\Phi|}
	\sum_{j=1}^{m} \frac{(-1)^{j+1}(\mathbf{s}_{\ell} \cdot \mathbf{x})^{2j}}{(2j)!}.
\]
This naturally motivates one to set $f_{m}(\mathbf{x}) := \dfrac{1}{p_m(\mathbf{x})}$ and consider the approximation of
$F_n(f)$ and $I_n(f)$ by
\begin{equation} \label{eq:Fnm-Infm-bare}
	F_n (f_m) := \sum_{\mathbf{t}_{j,k} \in D_n} f_m(\mathbf{t}_{j,k})
	\qquad
	\text{and}
	\qquad
	I_{n}(f_m) := \dfrac{1}{\Delta_n^2} \iint_{D_n} f_m(\mathbf{x}) \, d\mathbf{x}.
\end{equation}
These approximations do not generate any additional problems when $m=1$ since $p_1$ vanishes only at the origin in $\mathbb{R}^2$. Likewise, when $m>1$, $p_m$ vanishes at the origin but it may also have additional zeros in the rectangle $[-\pi,\pi] \times [-\pi,\pi]$, and therefore the domain $D_n$ in \eqref{eq:Fnm-Infm-bare} must be suitably restricted to a smaller rectangle. We set
\begin{equation} \label{eq:sbar}
	\overline{s} = \max_{1 \le \ell \le |\Phi|} \Vert \mathbf{s}_{\ell} \Vert ,
	\qquad
	(\overline{s} \ge 1 \text{ by \eqref{eq:s1-s2}}),
\end{equation}
and, given a fixed $\beta \in (0,1)$, we let
\[
	F_n^{\beta} (f) := \sum_{\mathbf{t}_{j,k} \in D_n^{\beta}} f(\mathbf{t}_{j,k}),
	\qquad
	I_n^{\beta}(f) := \frac{1}{\Delta_n^2} \iint_{D_n^{\beta}} f(\mathbf{x}) d\mathbf{x},
\]
and
\[
	F_n^{\beta} (f_m) := \sum_{\mathbf{t}_{j,k} \in D_n^{\beta}} f_m(\mathbf{t}_{j,k}),
	\qquad
	I_n^{\beta}(f_m) := \frac{1}{\Delta_n^2} \iint_{D_n^{\beta}} f_m(\mathbf{x}) d\mathbf{x},
\]
where
\[
	D_{n}^{\beta} := \bigcup_{|t_{j}|, |t_k| \le \frac{\sqrt{5(1-\beta)}}{\overline{s}} \atop (t_j,t_k) \ne (0,0)} X_{j,k}.
\]
The leading terms in the asymptotic expansions of these sums and integrals are all the same and this constitutes the first main result of this paper.

\begin{Thm} \label{thm:main1}
If $\Sigma_n \in \{ F_n(f), F_n^{\beta} (f), F_n(f_1), F_n^{\beta} (f_m) , 
I_n(f), I_n^{\beta}(f), I_n(f_1), I_n^{\beta}(f_m) \}$ for some fixed
$\beta \in (0,1)$ and $m \ge 1$, then
\[
	\Sigma_n = \frac{|\Phi|}{\pi \sqrt{\det(S^TS)}} \, n^2 \log n + \mathcal{O}(n^2)
\]
as $n \to \infty$ where
\begin{equation} \label{eq:S}
	S =
	\begin{bmatrix}
		s_{1,1} & \cdots & s_{1,\ell} & \cdots & s_{1,|\Phi|}
		\\
		s_{2,1} & \cdots & s_{2,\ell} & \cdots & s_{2,|\Phi|}
	\end{bmatrix}^T.
\end{equation}
\end{Thm} 

The second main result addresses the problem of obtaining the secondary main term in the asymptotic expansion of the sum $F_n(f)$.

\begin{Thm} \label{thm:main2}
As $n \to \infty$, we have
\begin{equation} \label{eq:simpleOlog}
	F_n(f) - I_n(f) + I_n(f_1) - F_n(f_1)
	= \mathcal{O}(\log n),
\end{equation}
and, for any fixed $\beta \in (0,1)$,
\begin{equation} \label{eq:Olog-O1}
	F_n(f) - I_n(f) + I_n^{\beta}(f_m) - F_n^{\beta}(f_m) = 
	\left\{ 	
		\begin{array}{cl}
			\mathcal{O}(\log n), & m = 1,
			\\
			\mathcal{O}(1), & m \ge 2.	
		\end{array}
	\right.
\end{equation}
\end{Thm}

We give the proofs of Theorems~\ref{thm:main1} and \ref{thm:main2} in \S\ref{sec:lemmata}.

The use of \eqref{eq:simpleOlog} for the triangular lattice reduces to our method in \cite{BEY}, whereas here we study
the asymptotic behavior of $I_n(f_1)$ and $F_n(f_1)$ for general planar lattices and thereby obtain the asymptotic behavior of the difference
$F_n(f) - I_n(f)$ upto an error of $\mathcal{O}(\log n)$. The integral $I_n(f)$ depends on the underlying lattice, and the evaluation may be carried out for a given lattice. Here we demonstrate this evaluation for two different examples in \S\ref{sec:square}.

Formula \eqref{eq:Olog-O1} provides the possibility of obtaining improved approximations when implemented for (the optimal choice) $m=2$, and this is left for future work.

In our analysis, the sum $F_n(f_1)$ will be related to the sum
\begin{equation} \label{eq:G_n}
	G_n := \sum_{j,k = 1}^{n-1} \big(\dfrac{1}{aj^2-bjk+ck^2} + \dfrac{1}{aj^2+bjk+ck^2}\big).
\end{equation}

\begin{Thm} \label{thm:Gn}
Let $aj^2+bjk+ck^2$ be a positive definite binary quadratic form with real coefficients and discriminant
\begin{align}
	d := b^2 - 4ac < 0.
\label{eq: 1.1}
\end{align}
As $n \to \infty$, we have
\begin{align} 
	G_n
	& = {2\pi \over \sqrt{|d|}} \log n
	+ \left[ {2\pi \over \sqrt{|d|}} \gamma - {\pi^2 \over 6} \big( {1 \over a} + {1 \over c} \big) - {4\pi \over \sqrt{|d|}} \log |\eta(\mu)| \right.
%\nonumber
	\label{eq:Gnasymp} \\
	& \qquad \qquad \hspace{1.7cm}
	\left. + {1 \over \sqrt{|d|}} \Big( \big( {\pi \over 2} - \arctan \big( {c-a \over \sqrt{d}} \big) \big) \log \big( {c \over a} \big) - C(a,b,c) \Big) \right]
\nonumber \\
  	&  + \left[ {1 \over a} + {1 \over c} - {\pi \over \sqrt{|d|}} \right] {1 \over n}
	+ \left[ {1 \over 2} \big( {1 \over a} + {1 \over c} \big)
	+ {1 \over 12} \big( {1 \over a-b+c} + {1 \over a+b+c} \big) - {\pi \over 6\sqrt{|d|}}	\right] {1 \over n^2}
	\nonumber \\
	& + \left[ {1 \over 6} \big( {1 \over a} + {1 \over c} \big) + {1 \over 12} \big( {1 \over a-b+c} + {1 \over a+b+c} \big)	\right] {1 \over n^3}
	+ \mathcal{O} \Big({\log n \over n^4}\Big) \nonumber
\end{align}
where
\begin{equation} \label{eq:mu}
	\mu = {-b + \sqrt{|d|} i \over 2c},
\end{equation}
$\eta$ is the Dedekind eta-function and, with $\operatorname{Cl}_2$ denoting the Clausen function,
\begin{multline}
	C(a,b,c) :=
	\mathrm{Cl}_2\Big(\pi -2\arctan\big({2a-b\over \sqrt{|d|}}\big)\Big)
	+\mathrm{Cl}_2\Big(\pi-2\arctan\big({2a+b\over \sqrt{|d|}}\big)\Big)
	\\
	+ \mathrm{Cl}_2 \Big(\pi-2\arctan\big({2c-b\over \sqrt{|d|}}\big)\Big) 
	+\mathrm{Cl}_2\Big(\pi-2\arctan\big({2c+b\over \sqrt{|d|}}\big)\Big).
	\label{eq:Cabc}
\end{multline}
\end{Thm}

\begin{Rem} \label{rem:Gn}
Since $G_n$ is invariant under the interchange of $a$ and $c$, this result must also have the same property.
The only seemingly non-symmetrical contribution in (\ref{eq:Gnasymp}) is $\displaystyle -{4 \pi \over \sqrt{|d|}} \log |\eta(\mu)|
+ {1 \over \sqrt{|d|}} {\pi \over 2} \log {c \over a} $. But if one does the interchange in these terms, the value doesn't change
by virtue of the relation $\displaystyle \eta\big(-{1\over \tau}\big) = (-i\tau)^{{1\over 2}}\eta(\tau)$ for any $\tau$ with $\Im \tau > 0$
(using the branch $1^{{1\over 2}} = 1$). The invariance when $b$ is replaced by $-b$ follows from 
$|\eta(\mu)| = |\eta(-\overline{\mu})|$.
\end{Rem}

We give the proof of Theorem~\ref{thm:Gn} in \S\ref{sec:Gn}. Theorem~\ref{thm:Gn} generalizes the calculations in \S 6 of \cite{BEY} which was for the case of the triangular lattice.

Concerning $F_n(f_1)$ and $I_n(f_1)$, we set 
\begin{equation} \label{eq:abc}
	a := \sum_{\ell=1}^{|\Phi|} s_{1,\ell}^2,
	\qquad
	b := 2\sum_{\ell=1}^{|\Phi|} s_{1,\ell} s_{2,\ell},
	\qquad
	c := \sum_{\ell=1}^{|\Phi|} s_{2,\ell}^2,
\end{equation}
(so that $a,b,c \in \mathbb{Z}$). By construction, $d = -4 \det(S^TS)$ (see \eqref{eq:S}).
With this notation, the asymptotic behavior of $F_n(f_1)$ and $I_n(f_1)$ are given below.
\begin{Thm} \label{thm:Fnf1}
As $n \to \infty$, we have
\begin{align}
	& F_n(f_1)
	= {2|\Phi| \over \pi \sqrt{|d|}} n^2 \log n 
	+ {|\Phi| \over \pi^2 \sqrt{|d|}} 
	\Bigg[ 2\pi \big( \gamma - \log 2 \big) - 4\pi \log |\eta(\mu)| 
	\nonumber \\
	& \hspace{4.9cm} + \big( {\pi \over 2} - \arctan \big( {c-a \over \sqrt{|d|}} \big) \big) \log \big( {c \over a} \big) - C(a,b,c) \Bigg] n^2
	\nonumber \\ 
	& + {|\Phi| \over \pi^2} \Bigg[ {1 \over 3} \Big( {1 \over a-b+c} + {1 \over a+b+c} + {\pi \over \sqrt{|d|}} \Big)
	- {1+(-1)^n \over 2} \Big( {\pi \over \sqrt{|d|}} +{2 \over a-b+c} \Big) \Bigg]
	+ \mathcal{O} \big( {\log n \over n^2} \big).
	\nonumber \\
	\label{eq:Fnf1}
\end{align}
%where $\mu$ and $C(a,b,c)$ are as in \eqref{eq:mu} and \eqref{eq:Cabc} respectively.
\end{Thm}
\begin{Thm} \label{thm:Inf1}
If $n$ is odd,
\[
	I_n(f_1) = \frac{2|\Phi|}{\pi \sqrt{|d|}} \, n^2 \log n. 
\]
If $n$ is even, as $n \to \infty$,
%\[
%	I_n(f_1) = \frac{2|\Phi|}{\pi \sqrt{|d|}} \, \big( n^2 \log n - {1 \over 2} \big)
%	- {2 |\Phi| \over \pi^2} \, {1 \over a-b+c} + \mathcal{O} \big( {1 \over n^2} \big).
%\]
\[
	I_n(f_1) = \frac{2|\Phi|}{\pi \sqrt{|d|}} \, n^2 \log n
	- \big( \frac{|\Phi|}{\pi\sqrt{|d|}} + \frac{2|\Phi|}{\pi^2}{1 \over a-b+c} \big) + \mathcal{O} \big( {1 \over n^2} \big).
\]
\end{Thm}

The proof of Theorem~\ref{thm:Fnf1} is given in \S\ref{sec:Fnf1}, and that of Theorem~\ref{thm:Inf1} is in \S\ref{sec:lemmata}.

\section{Lemmata and Proofs of Theorems \ref{thm:main1}, \ref{thm:main2} and \ref{thm:Inf1}} \label{sec:lemmata}

First, we study the function $\psi$ in \eqref{eq:psi} and its partial derivatives
\begin{equation} \label{eq:partialderivatives}
	\partial_k \psi(\mathbf{x}) = \frac{1}{|\Phi|} \sum\limits_{\ell=1}^{|\Phi|} s_{k, \ell} \sin(\mathbf{s}_{\ell} \cdot \mathbf{x}) 
	\quad
	\text{and}
	\quad
	\partial_k^2 \psi(\mathbf{x}) = \frac{1}{|\Phi|} \sum\limits_{\ell=1}^{|\Phi|} s_{k, \ell}^2 \cos(\mathbf{s}_{\ell} \cdot \mathbf{x}),
\end{equation}
where $\displaystyle{\partial_k = {\partial \over \partial x_k}}$, for $k=1,2$.

\begin{Lem} \label{lemma:phi}
For all $\mathbf{x} \in \mathbb{R}^2$, we have
\begin{enumerate}
\item $|\psi(\mathbf{x})| \le \frac{1}{2|\Phi|} \Vert \mathbf{x} \Vert^2 \sum\limits_{\ell=1}^{|\Phi|} \Vert \mathbf{s}_{\ell} \Vert^2 \le \frac{1}{2} \overline{s}^2 \Vert \mathbf{x} \Vert^2$.
\item $|\partial_k \psi(\mathbf{x})| \le \frac{1}{|\Phi|}\Vert \mathbf{x}\Vert \sum\limits_{\ell=1}^{|\Phi|} |s_{k, \ell}| \Vert \mathbf{s}_{\ell} \Vert \le \overline{s}^2 \Vert \mathbf{x}\Vert$.
\item $|\partial_k^2 \psi(\mathbf{x})| \le \frac{1}{|\Phi|} \sum\limits_{\ell=1}^{|\Phi|} s_{k, \ell}^2 \le \overline{s}^2$.
\end{enumerate}
\end{Lem} 
\begin{proof} These follow from the triangle and Cauchy-Schwarz inequalities, and the inequalities $|\sin \theta| \le |\theta|$ 
and $|\cos \theta| \le 1$ applied to \eqref{eq:psi} and \eqref{eq:partialderivatives}.
\end{proof}

\begin{Lem} \label{lemma:phi-lower}
Given $0 < \beta < 2\pi$, there exists a constant $C_1({\beta}) >0$ such that
$\psi(\mathbf{x}) \ge C_1({\beta}) \, \Vert \mathbf{x} \Vert^2$ for all $\mathbf{x} \in [-2\pi+\beta,2\pi-\beta] \times [-2\pi+\beta,2\pi-\beta]$.
\end{Lem}
\begin{proof}
For any $\mathbf{x} = (x_1,x_2) \in \mathbb{R}^2$, we have
\[
	\psi(\mathbf{x})
	\ge \frac{1}{|\Phi|} \sum_{\ell = 1}^{2} \big( 1- \cos(\mathbf{s}_{\ell} \cdot \mathbf{x}) \big)
	= \frac{1}{|\Phi|} \big( 2 - \cos x_1 - \cos x_2 \big),
\]
and by Taylor's theorem with the fourth order remainder at the origin
\[
	2 - \cos x_1 - \cos x_2
	= \frac{x_1^2+x_2^2}{2}-\frac{x_1^4 \, \cos y_1 + x_2^4 \, \cos y_2}{24}
\]
for some $\mathbf{y} = (y_1,y_2)$ on the line segment between $\mathbf{x}$ and the origin. Therefore
\begin{align*}
	\psi(\mathbf{x})
	& \ge \frac{1}{|\Phi|} \Big( \frac{x_1^2+x_2^2}{2}-\frac{x_1^4 \, \cos y_1 + x_2^4 \, \cos y_2}{24} \Big)
	\ge \frac{1}{|\Phi|} \Big( \frac{x_1^2+x_2^2}{2}-\frac{x_1^4 + x_2^4}{24} \Big)
	\\
	& \ge \frac{1}{|\Phi|} \Big( \frac{x_1^2+x_2^2}{2}-\frac{(x_1^2 + x_2^2)^2}{24} \Big)
	= \frac{12 -\Vert \mathbf{x} \Vert^2}{24 \, |\Phi|} \, \Vert \mathbf{x} \Vert^2.
\end{align*}
In particular, if $\Vert \mathbf{x} \Vert \le \sqrt{6}$, then $\psi(\mathbf{x}) \ge \frac{1}{4|\Phi|} \Vert \mathbf{x} \Vert^2$. Further, since $\psi$ is continuous and does not vanish on the compact set $[-2\pi+\beta,2\pi-\beta] \times [-2\pi+\beta,2\pi-\beta] \backslash B(0,\sqrt{6})$, there exists a constant $C_{2}(\beta)>0$ such that $\psi(\mathbf{x}) \ge C_2(\beta) \Vert \mathbf{x} \Vert^2$ for all $\mathbf{x}$ in this set. The result now follows by setting
$C_1({\beta})  = \min \{ C_2(\beta), \frac{1}{4|\Phi|} \}$.
\end{proof}

\begin{Cor} \label{cor:est-f2}
Given $0 < \beta < 2\pi$, there exists a constant $C_3({\beta}) >0$ such that
$|\partial_{k}^{2} f(\mathbf{x})| \le \dfrac{C_3({\beta})}{\norm{\mathbf{x}}^4}$ for all
$\mathbf{x} \in [-2\pi+\beta,2\pi-\beta] \times [-2\pi+\beta,2\pi-\beta] \backslash \{ \mathbf{0} \}$.
\end{Cor}
\begin{proof}
Since $\partial_{k}^{2} f(\mathbf{x}) = \displaystyle{\frac{2 (\partial_{k} \psi(\mathbf{x}))^2-\psi(\mathbf{x}) \, \partial_{k}^{2} \psi(\mathbf{x})}{(\psi(\mathbf{x}))^{3}}}$, we have
\[
	|\partial_{k}^{2} f(\mathbf{x})|
	\le \frac{2 |\partial_{k} \psi(\mathbf{x})|^2 + |\psi(\mathbf{x})| \, |\partial_{k}^2 \psi(\mathbf{x})|}{|\psi(\mathbf{x})|^{3}}
	\le \dfrac{{5 \over 2}\, \overline{s}^4 \, \Vert \mathbf{x}\Vert^2}{C_1(\beta)^3 \, \Vert \mathbf{x}\Vert^6}
	= \dfrac{C_3({\beta})}{\norm{\mathbf{x}}^4}
\]
by Lemmas \ref{lemma:phi} and \ref{lemma:phi-lower}.
\end{proof}

Next we derive the Taylor series approximations of $\psi$ and of its first and second order partial derivatives around the origin.

\begin{Lem} \label{Lemma:psi-and-derivatives}
For all $m \ge 0$, all $\mathbf{x} \in \mathbb{R}^2$, and $k = 1,2$, we have
\[
	\psi(\mathbf{x}) = \frac{1}{|\Phi|} \sum_{\ell = 1}^{|\Phi|}
	\Big\{
		\sum_{j=1}^{m} \frac{(-1)^{j+1}(\mathbf{s}_{\ell} \cdot \mathbf{x})^{2j}}{(2j)!}
		+ \frac{(-1)^{m}(\mathbf{s}_{\ell} \cdot \mathbf{x})^{2m+2}}{(2m+2)!} \cos (\mathbf{s}_{\ell} \cdot \mathbf{p}) 
	\Big\},
\]
\[
	\partial_k \psi(\mathbf{x}) = \frac{1}{|\Phi|} \sum_{\ell = 1}^{{|\Phi|}} s_{k,\ell}
	\Big\{
		\sum_{j=1}^{m} \frac{(-1)^{j+1}(\mathbf{s}_{\ell} \cdot \mathbf{x})^{2j-1}}{(2j-1)!}
		+ \frac{(-1)^{m}(\mathbf{s}_{\ell} \cdot \mathbf{x})^{2m+1}}{(2m+1)!} \cos (\mathbf{s}_{\ell} \cdot \mathbf{q}) 
	\Big\},
\]
\[
	\partial_k^2 \psi(\mathbf{x}) = \frac{1}{{|\Phi|}} \sum_{\ell = 1}^{{|\Phi|}} s_{k,\ell}^2
	\Big\{
		1+ \sum_{j=1}^{m} \frac{(-1)^{j}(\mathbf{s}_{\ell} \cdot \mathbf{x})^{2j}}{(2j)!}
		+ \frac{(-1)^{m+1}(\mathbf{s}_{\ell} \cdot \mathbf{x})^{2m+2}}{(2m+2)!} \cos (\mathbf{s}_{\ell} \cdot \mathbf{r}) 
	\Big\}
\]
for some $\mathbf{p},\mathbf{q},\mathbf{r} \in \mathbb{R}^2$ on the line segment between $\mathbf{x}$
and the origin which depend on $m$ and $\mathbf{x}$.
\end{Lem}
\begin{proof}

Taylor's theorem about the origin states
\begin{multline*}
	f(\mathbf{x})
	= \sum_{j=0}^{n} \frac{1}{j!} \sum_{k=0}^{j} \binom{j}{k} x_1^{j-k} x_2^{k} \ (\partial_1^{j-k} \partial_2^k f)(\mathbf{0})
	\\
	+ \frac{1}{(n+1)!} \sum_{k=0}^{n+1} \binom{n+1}{k} x_1^{n+1-k} x_2^k \ \partial_1^{n+1-k} \partial_2^kf(\mathbf{y})
\end{multline*}
for some $\mathbf{y} = (y_1,y_2)$ on the line segment between $\mathbf{x} = (x_1,x_2)$ and the origin. On the other hand, for
$\mathbf{s} = (s_1,s_2) \in \mathbb{R}^2$, direct differentiation yields
\[
	\partial_1^{j-k} \partial_2^{k} \cos (\mathbf{s} \cdot \mathbf{x})
	= s_1^{j-k} s_2^{k} \cos(\frac{j \pi}{2} + \mathbf{s} \cdot \mathbf{x} ), 
	\qquad
	(0 \le k \le j),
\]
\[
	\partial_1^{j-k} \partial_2^{k} \sin (\mathbf{s} \cdot \mathbf{x})
	= s_1^{j-k} s_2^{k} \sin(\frac{j \pi}{2} + \mathbf{s} \cdot \mathbf{x}), 
	\qquad
	(0 \le k \le j).
\]
Using these identities, we obtain for $n \ge 0$
\[
	\psi(\mathbf{x}) = -\frac{1}{{|\Phi|}} \sum_{\ell = 1}^{{|\Phi|}}
	\Big\{
		\sum_{j=1}^{n} \frac{(\mathbf{s}_{\ell} \cdot \mathbf{x})^{j}}{j!} \cos \frac{j\pi}{2}
		+ \frac{(\mathbf{s}_{\ell} \cdot \mathbf{x})^{n+1}}{(n+1)!} \cos ( \frac{n+1}{2}\pi + \mathbf{s}_{\ell} \cdot \mathbf{p}) 
	\Big\},
\]
\[
	\partial_{k} \psi(\mathbf{x})
	= \frac{1}{{|\Phi|}} \sum_{\ell = 1}^{{|\Phi|}} s_{k,\ell}
	\Big\{
		\sum_{j=1}^{n} \frac{(\mathbf{s}_{\ell} \cdot \mathbf{x})^j}{j!} \sin \frac{j\pi}{2}
		+ \frac{(\mathbf{s}_{\ell} \cdot \mathbf{x})^{n+1}}{(n+1)!} \sin (\frac{n+1}{2}\pi + \mathbf{s}_{\ell} \cdot \mathbf{q}) 
	\Big\},
\]
\[
	\partial^2_{k} \psi(\mathbf{x})
	= \frac{1}{{|\Phi|}} \sum_{\ell = 1}^{{|\Phi|}} s_{k,\ell}^2
	\Big\{
		1+ \sum_{j=1}^{n} \frac{(\mathbf{s}_{\ell} \cdot \mathbf{x})^j}{j!} \cos \frac{j\pi}{2}
		+ \frac{(\mathbf{s}_{\ell} \cdot \mathbf{x})^{n+1}}{(n+1)!} \cos (\frac{n+1}{2}\pi + \mathbf{s}_{\ell} \cdot \mathbf{r}) 
	\Big\},
\]
for some $\mathbf{p},\mathbf{q},\mathbf{r} \in \mathbb{R}^2$ on the line segment between $\mathbf{x}$
and the origin which depend on $n$ and $\mathbf{x}$.
In these identities, taking $n = 2m +1$ for $\psi$ and $\partial_k^2 \psi$,
and $n = 2m$ for $\partial_k \psi$ ($m\ge0$) delivers the desired results.
\end{proof}

Note that, for $m \ge 1$,
\[
	\partial_k p_{m}(\mathbf{x})
	= \frac{1}{|\Phi|} \sum_{\ell = 1}^{|\Phi|} s_{k, \ell} \sum_{j=1}^{m} \frac{(-1)^{j+1}(\mathbf{s}_{\ell} \cdot \mathbf{x})^{2j-1}}{(2j-1)!},
\]
and
\[
	\partial_k^2 p_{m}(\mathbf{x})
%	& = \frac{1}{d} \sum_{\ell = 1}^d s_{k\ell}^2 \left\{ 1+ \sum_{j=2}^{m} \frac{(-1)^{j+1}(\mathbf{s}_{\ell} \cdot \mathbf{x})^{2j-2}}{(2j-2)!} \right\}
%	\\
%	&
	= \frac{1}{|\Phi|} \sum_{\ell = 1}^{|\Phi|}
	s_{k,\ell}^2 \Big\{ 1+ \sum_{j=1}^{m-1} \frac{(-1)^{j}(\mathbf{s}_{\ell} \cdot \mathbf{x})^{2j}}{(2j)!} \Big\}.
\]
The preceding Lemma therefore implies the following identities.

\begin{Cor} For all $m \ge 1$, all $\mathbf{x} \in \mathbb{R}^2$, and $k=1,2$, we have
\begin{equation} \label{eq:phi-pm}
	\psi(\mathbf{x})
	= p_m(\mathbf{x})
	+ \frac{1}{|\Phi|} \frac{(-1)^{m}}{(2m+2)!}
	\sum_{\ell = 1}^{|\Phi|} (\mathbf{s}_{\ell} \cdot \mathbf{x})^{2m+2} \cos (\mathbf{s}_{\ell} \cdot \mathbf{p}),
\end{equation}
\[
	\partial_k \psi(\mathbf{x})
	= \partial_k p_m(\mathbf{x})
	+ \frac{1}{|\Phi|} \frac{(-1)^m}{(2m+1)!}
	\sum_{\ell = 1}^{|\Phi|} s_{k,\ell} (\mathbf{s}_{\ell} \cdot \mathbf{x})^{2m+1} \cos (\mathbf{s}_{\ell} \cdot \mathbf{q}),
\]
\[
	\partial_k^2 \psi(\mathbf{x})
	= \partial_k^2 p_m(\mathbf{x})
	+ \frac{1}{|\Phi|} \frac{(-1)^{m}}{(2m)!}
	\sum_{\ell = 1}^{|\Phi|} s_{k,\ell}^2 (\mathbf{s}_{\ell} \cdot \mathbf{x})^{2m} \cos (\mathbf{s}_{\ell} \cdot \mathbf{r}) 
\]
for some $\mathbf{p},\mathbf{q},\mathbf{r} \in \mathbb{R}^2$ on the line segment between $\mathbf{x}$ and the origin which depend on $m$ and $\mathbf{x}$, and, where relevant, on $k$.
\end{Cor}

These, in return, yield the following estimates.
\begin{Cor} \label{cor:phi-p}
For all $m \ge 1$, all $\mathbf{x} \in \mathbb{R}^2$, and $k=1,2$, we have
\[
	|\psi(\mathbf{x}) - p_m(\mathbf{x})|
	\le \frac{1}{|\Phi|} \frac{1}{(2m+2)!} \Vert \mathbf{x} \Vert^{2m+2}
	\sum_{\ell = 1}^{|\Phi|} \Vert \mathbf{s}_{\ell} \Vert^{2m+2}
	\le \frac{\overline{s}^{2m+2}}{(2m+2)!} \Vert \mathbf{x} \Vert^{2m+2},
\]
\[
	|\partial_k \psi(\mathbf{x}) - \partial_k p_m(\mathbf{x})|
	\le \frac{1}{|\Phi|} \frac{1}{(2m+1)!} \Vert \mathbf{x} \Vert^{2m+1}
	\sum_{\ell = 1}^{|\Phi|} |s_{k,\ell}| \Vert \mathbf{s}_{\ell} \Vert^{2m+1}
	\le \frac{\overline{s}^{2m+2}}{(2m+1)!} \Vert \mathbf{x} \Vert^{2m+1},
\]
\[
	|\partial_k^2 \psi(\mathbf{x}) - \partial_k^2 p_m(\mathbf{x})|
	\le \frac{1}{|\Phi|} \frac{1}{(2m)!} \Vert \mathbf{x} \Vert^{2m}
	\sum_{\ell = 1}^{|\Phi|} s_{k,\ell}^2 \Vert \mathbf{s}_{\ell} \Vert^{2m}
	\le \frac{\overline{s}^{2m+2}}{(2m)!} \Vert \mathbf{x} \Vert^{2m} .
\]
\end{Cor}
\begin{proof}
Follows from the preceding Corollary in conjunction with the triangle and Cauchy-Schwarz inequalities.
\end{proof}

Now we return to the problem of estimating $p_m,\partial_k p_m, \partial_k^2 p_m$. As above, the following estimates
are immediate from the triangle and Cauchy-Schwarz inequalities.

\begin{Lem} \label{lemma:p}
Given $C_4>0$, there holds
\[
	|p_{m}(\mathbf{x})|
	\le \frac{1}{|\Phi|} \Vert \mathbf{x} \Vert^2 \sum_{\ell = 1}^{|\Phi|} \sum_{j=1}^{m} \frac{\Vert \mathbf{s}_{\ell} \Vert^{2j} C_4^{2j-2}}{(2j)!}
	\le \frac{\cosh(\overline{s} \, C_4)-1}{C_4^2} \Vert \mathbf{x} \Vert^2,
\]
\[
	|\partial_k p_{m}(\mathbf{x})|
	\le \frac{1}{|\Phi|} \Vert \mathbf{x} \Vert
	\sum_{\ell = 1}^{|\Phi|} |s_{k,\ell}| \sum_{j=1}^{m} \frac{\Vert \mathbf{s}_{\ell} \Vert^{2j-1} C_4^{2j-2}}{(2j-1)!}
	\le \frac{\overline{s} \, \sinh(\overline{s} \, C_4)}{C_4} \Vert \mathbf{x} \Vert,
\]
\[
	|\partial_k^2 p_{m}(\mathbf{x})|
	\le \frac{1}{|\Phi|} \sum_{\ell = 1}^{|\Phi|} s_{k,\ell}^2 \Big\{ 1+ \sum_{j=1}^{m-1} \frac{\Vert \mathbf{s}_{\ell} \Vert^{2j} C_4^{2j}}{(2j)!} \Big\}
	\le \overline{s}^2 \cosh(\overline{s} \, C_4),
\]
for all $m \ge 1$, all $\mathbf{x} \in \mathbb{R}^2$ with $\Vert \mathbf{x} \Vert \le C_4$, and $k = 1,2$.
\end{Lem}

\begin{Lem} \label{lemma:p-upper-lower}
We have:
\begin{enumerate}
\item For all $\mathbf{x} \in \mathbb{R}^2$, $p_1(\mathbf{x}) \ge \dfrac{1}{2|\Phi|} \Vert \mathbf{x}\Vert^2$.
\item If $0\le \beta \le 1$ and $\Vert \mathbf{x} \Vert \le \dfrac{\sqrt{12(1-\beta)}}{\overline{s}}$,
then $p_m(\mathbf{x}) \ge \dfrac{\beta}{2|\Phi|} \Vert \mathbf{x} \Vert^2$ holds for all $m \ge 1$.
\end{enumerate}
\end{Lem}
\begin{proof}
(1) For any $\mathbf{x} \in \mathbb{R}^2$, by \eqref{eq:s1-s2} we have
\[
	p_1(\mathbf{x})
	= \frac{1}{|\Phi|} \sum_{\ell = 1}^{|\Phi|}  \frac{(\mathbf{s}_{\ell} \cdot \mathbf{x})^{2}}{2!}
	\ge \frac{1}{|\Phi|} \sum_{\ell = 1}^2  \frac{(\mathbf{s}_{\ell} \cdot \mathbf{x})^{2}}{2!}
	= \frac{1}{2|\Phi|} \Vert \mathbf{x}\Vert^2.
\]
(2) For $m=1$, this follows from (1). If $m\ge 2$ is even, say $m=2n$, then
\begin{align*}
	p_m(\mathbf{x})
	& = \frac{1}{|\Phi|} \sum_{\ell =1}^{|\Phi|}
	\sum_{k=1}^{n}
	\Big( \frac{(\mathbf{s}_{\ell} \cdot \mathbf{x})^{4k-2}}{(4k-2)!}-\frac{(\mathbf{s}_{\ell} \cdot \mathbf{x})^{4k}}{(4k)!} \Big)
	= \frac{1}{|\Phi|} \sum_{\ell =1}^{|\Phi|} \sum_{k=1}^{n} \frac{(\mathbf{s}_{\ell} \cdot \mathbf{x})^{4k-2}}{(4k-2)!}
	\big( 1-\frac{(\mathbf{s}_{\ell} \cdot \mathbf{x})^{2}}{4k(4k-1)} \big)
	\\
	&
	\ge \frac{1}{|\Phi|} \sum_{\ell =1}^{|\Phi|} \sum_{k=1}^{n} \frac{(\mathbf{s}_{\ell} \cdot \mathbf{x})^{4k-2}}{(4k-2)!}
	\big( 1-\frac{ \Vert \mathbf{s}_{\ell} \Vert^2 \Vert \mathbf{x} \Vert^2}{12} \big)
	\ge \frac{1}{|\Phi|} \sum_{\ell =1}^{|\Phi|} \sum_{k=1}^{n} \frac{(\mathbf{s}_{\ell} \cdot \mathbf{x})^{4k-2}}{(4k-2)!} \big( 1-\frac{\overline{s}^2 \Vert \mathbf{x} \Vert^2}{12} \big)
	\\
	& \ge \frac{\beta}{|\Phi|} \sum_{\ell =1}^{|\Phi|} \sum_{k=1}^{n} \frac{(\mathbf{s}_{\ell} \cdot \mathbf{x})^{4k-2}}{(4k-2)!}
	\ge \frac{\beta}{|\Phi|} \sum_{\ell =1}^{|\Phi|} \frac{(\mathbf{s}_{\ell} \cdot \mathbf{x})^{2}}{2!}
	\ge \frac{\beta}{|\Phi|} \sum_{\ell =1}^{2} \frac{(\mathbf{s}_{\ell} \cdot \mathbf{x})^{2}}{2!}
	= \frac{\beta}{2|\Phi|} \Vert \mathbf{x} \Vert^2.
\end{align*}
If $m \ge 3$ is odd, then $m-1$ is even, and therefore the previous step yields
\[
	p_m(\mathbf{x})
	= p_{m-1}(\mathbf{x}) + \frac{1}{|\Phi|} \sum_{\ell = 1}^{|\Phi|} \frac{(\mathbf{s}_{\ell} \cdot \mathbf{x})^{2m}}{(2m)!}
	\ge p_{m-1}(\mathbf{x})
	\ge \frac{\beta}{2|\Phi|} \Vert \mathbf{x} \Vert^2.
\]
This finishes the proof.
\end{proof}

\begin{Cor} \label{cor:fm-2} We have:
\begin{enumerate}
\item Given $C_4 >0$, there exists a constant $C'_4>0$ that depends only on $C_4$ such that $|\partial_{k}^{2} f_1(\mathbf{x})| \le \dfrac{C'_4}{\norm{\mathbf{x}}^4}$ holds for $0 < \Vert \mathbf{x} \Vert \le C_4$.
\item Given a fixed $\beta \in (0,1)$, there exists a constant $C_5(\beta) >0$ such that $|\partial_{k}^{2} f_m(\mathbf{x})| \le \dfrac{C_5(\beta)}{\norm{\mathbf{x}}^4}$ holds for all $m \ge 1$ for $0< \Vert \mathbf{x} \Vert \le \dfrac{\sqrt{12(1-\beta)}}{\overline{s}}$.
\end{enumerate}
\end{Cor}
\begin{proof}
Since $\partial_{k}^{2} f_m(\mathbf{x}) = \dfrac{2 (\partial_{k} p_m(\mathbf{x}))^2-p_m(\mathbf{x}) \, \partial_{k}^{2} p_m(\mathbf{x})}{(p_m(\mathbf{x}))^{3}}$, we have
\[
	|\partial_{k}^{2} f_m(\mathbf{x})|
	\le \frac{2 |\partial_{k} p_m(\mathbf{x})|^2 + |p_m(\mathbf{x})| \, |\partial_{k}^2 p_m(\mathbf{x})|}{|p_m(\mathbf{x})|^{3}},
	\qquad (m \ge 1).
\]
%Given $C_4 >0$, Lemmas \ref{lemma:p} and \ref{lemma:p-upper-lower} therefore entail
%\begin{align*}
%	|\partial_{k}^{2} f_1(\mathbf{x})|
%	& \le \dfrac{2 \, \dfrac{\overline{s}^2}{C_4^2} \, \sinh^2 (\overline{s} \, C_4) \Vert \mathbf{x} \Vert^2
%	+ \dfrac{\cosh (\overline{s} \, C_4)-1}{C_4^2} \, \Vert \mathbf{x} \Vert^2 \, \overline{s}^2 \cosh( \overline{s} \, C_4)}{\dfrac{1}{(2 \, |\Phi|)^3} \, \Vert \mathbf{x} \Vert^6}
%	\\
%	& = \dfrac{8 \, |\Phi|^3 \, \overline{s}^2}{C_4^2} \big[ 3\cosh^2 (\overline{s}C_4) - \cosh(\overline{s}C_4) -2 \big]
%	\dfrac{1}{\norm{\mathbf{x}}^4}
%\end{align*}
%whenever $0 < \Vert \mathbf{x} \Vert \le C_4$. This establishes the first part.
%
%As for the second part, given $\beta \in (0,1)$, taking $C_4 = \dfrac{\sqrt{12(1-\beta)}}{\overline{s}}$ and setting
Lemmas \ref{lemma:p} and \ref{lemma:p-upper-lower} imply
\begin{align*}
	|\partial_{k}^{2} f_m(\mathbf{x})|
	& \le \dfrac{2 \, \dfrac{\overline{s}^2}{C_4^2} \, \sinh^2 (\overline{s} \, C_4) \Vert \mathbf{x} \Vert^2
	+ \dfrac{\cosh (\overline{s} \, C_4)-1}{C_4^2} \, \Vert \mathbf{x} \Vert^2 \ \overline{s}^2 \, \cosh( \overline{s} \, C_4)}{\dfrac{\delta^3}{(2 \, |\Phi|)^3} \, \Vert \mathbf{x} \Vert^6}
	\\
	& = \dfrac{8 \, |\Phi|^3 \, \overline{s}^2}{\delta^3 \, C_4^2} 
	\left[
		3 \cosh^2 (\overline{s} \, C_4) - \cosh(\overline{s} \, C_4) -2
	\right] \,
	\dfrac{1}{\Vert \mathbf{x} \Vert^4},
\end{align*}
where
\[
	\delta = 
	\left\{
		\begin{array}{ll}
			1, & \text{if } m =1,
			\\
		 	\beta, & \text{if } m > 1,
		\end{array}
	\right.
\]
for $0 < \Vert \mathbf{x} \Vert \le C_4$ with $C_4 = \dfrac{\sqrt{12(1-\beta)}}{\overline{s}}$ for
$\beta \in (0,1)$ if $m \ge 2$. This establishes the first part. When $m \ge 2$,
\begin{align*}
	|\partial_{k}^{2} f_m(\mathbf{x})|
	& = \dfrac{8 \, |\Phi|^3 \, \overline{s}^4}{\delta^3} \,
	\dfrac{3 \cosh^2 (\sqrt{12(1-\beta)}) - \cosh(\sqrt{12(1-\beta)}) -2}{12(1-\beta)} \,
	\dfrac{1}{\Vert \mathbf{x} \Vert^4}
	\\
	& < \dfrac{8 \, |\Phi|^3 \, \overline{s}^4}{\delta^3} \,
	\dfrac{3 \cosh^2 (\sqrt{12}) - \cosh(\sqrt{12}) -2}{12} \,
	\dfrac{1}{\Vert \mathbf{x} \Vert^4}.
\end{align*}
This gives the second part.
\end{proof}

\begin{Lem} \label{lemma:der-est}
Given a fixed $\beta \in (0,1)$, there exists a constant $C_6(\beta)>0$ such that,
for all $m \ge 1$, and all $\mathbf{x} \in \mathbb{R}^2$ with
$0<\Vert \mathbf{x} \Vert \le \frac{\sqrt{12(1-\beta)}}{\overline{s}}$, we have
\[
	|\partial_k^2 h_m(\mathbf{x})| \le C_6(\beta) \Vert \mathbf{x} \Vert^{2m-4}
\]
for $k=1,2$ where $h_m(\mathbf{x}) = f(\mathbf{x})-f_m(\mathbf{x})$.
\end{Lem}
\begin{proof}
We write $\partial_k^2 h_m$ as
\begin{multline*}
	\partial_k^2 h_m
	= \frac{2(\partial_k \psi)^2 (p_m-\psi) (p_m^2+p_m\psi+\psi^2)}{\psi^3p_m^3}
	+ \frac{2(\partial_k \psi - \partial_k p_m)(\partial_k \psi + \partial_k p_m)}{p_m^3}
	\\
	+ \frac{\partial_k^2 \psi (\psi-p_m) (\psi+p_m)}{\psi^2p_m^2}
	+ \frac{\partial_k^2 p_m - \partial_k^2 \psi}{p_m^2},
\end{multline*}
and we estimate
\begin{multline*}
	|\partial_k^2 h_m|
	\le \frac{2(\partial_k \psi)^2 |p_m-\psi| (p_m^2+p_m\psi+\psi^2)}{\psi^3p_m^3}
	+ \frac{2|\partial_k \psi - \partial_k p_m|(|\partial_k \psi| + |\partial_k p_m|)}{p_m^3}
	\\
	+ \frac{|\partial_k^2 \psi| |\psi-p_m| (\psi+p_m)}{\psi^2p_m^2}
	+ \frac{|\partial_k^2 p_m - \partial_k^2 \psi|}{p_m^2}.
\end{multline*}
Thus the result follows from Lemmas \ref{lemma:phi}, \ref{lemma:phi-lower}, \ref{lemma:p}, \ref{lemma:p-upper-lower} and Corollary \ref{cor:phi-p}.
\end{proof}

In comparing the sums and integrals mentioned in \S\ref{sec:mainresults}, we shall also need the following general estimates.
\begin{Lem} \label{lemma:radialsum}
Given $\alpha \in \mathbb{R}$, we have, as $n \to \infty$,
\[
	\Delta_n^2 \sum_{\mathbf{t}_{j,k} \in D_n} \max_{\mathbf{x} \in X_{j,k}} \frac{1}{\Vert \mathbf{x} \Vert^\alpha}
	= \left\{
		\begin{array}{cl}
			\mathcal{O}(n^{\alpha-2}), & \alpha > 2, \\
			\mathcal{O}(\log n), & \alpha = 2, \\
			\mathcal{O}(1), & \alpha < 2.
		\end{array}
	\right.
\]
\end{Lem}
\begin{proof}
When $\alpha > 0$, for all sufficiently large $n$, setting $\mathbf{x}_{j,k} = (x_j,x_k)$, we have by \eqref{eq:D_n}
\begin{align*}
	\Delta_n^2 \sum_{\mathbf{t}_{j,k} \in D_n} \max_{\mathbf{x} \in X_{j,k}} \frac{1}{\Vert \mathbf{x} \Vert^\alpha}
	& \le
	4\Delta_n^2 \big( \sum_{\mathbf{t}_{j,k} \in (0,\pi]^2} \max_{\mathbf{x} \in X_{j,k}} \frac{1}{\Vert \mathbf{x} \Vert^\alpha}
	+ \sum_{t_{j} \in (0,\pi]} \max_{\mathbf{x} \in X_{j,0}} \frac{1}{\Vert \mathbf{x} \Vert^\alpha} \big)
	\\
	& = 4\Delta_n^2 
	\big( \sum_{\mathbf{t}_{j,k} \in (0,\pi]^2} \frac{1}{\Vert \mathbf{x}_{j,k} \Vert^\alpha}
	+ \sum_{t_{j} \in (0,\pi]} \frac{1}{x_j^\alpha} \big)
	\\
	& \le 4\Delta_n^2 
	\big( \sum_{\mathbf{t}_{j,k} \in (t_1,\pi]^2} \frac{1}{\Vert \mathbf{x}_{j,k} \Vert^\alpha}
	+ 2 \sum_{t_{j} \in (0,\pi]} \frac{1}{\Vert \mathbf{x}_{j,1} \Vert^\alpha}
	+ \sum_{t_{j} \in (0,\pi]} \frac{1}{x_j^\alpha} \big)
	\\
	& \le 4\Delta_n^2 
	\big( \sum_{\mathbf{t}_{j,k} \in (t_1,\pi]^2} \frac{1}{\Vert \mathbf{x}_{j,k} \Vert^\alpha}
	+ 3 \sum_{t_{j} \in (0,\pi]} \frac{1}{x_j^\alpha} \big)
	\\
	& \le 4\Delta_n^2 
	\big( \sum_{t_1 < t_j,t_k \atop \Vert \mathbf{t}_{j,k} \Vert \le \sqrt{2} \pi} \frac{1}{\Vert \mathbf{x}_{j,k} \Vert^\alpha}
	+ \frac{3}{x_1^{\alpha}}
	+ 3 \sum_{t_{j} \in (t_1,\pi]} \frac{1}{x_j^\alpha} \big)
	\\
	& = 4\Delta_n^2 
	\big( \dfrac{1}{\Delta_n^2}
	\sum_{t_1 < t_j,t_k \atop \Vert \mathbf{t}_{j,k} \Vert \le \sqrt{2} \pi} \frac{1}{\Vert \mathbf{x}_{j,k} \Vert^\alpha} \, \Delta_n^2
	+ \frac{3}{x_1^{\alpha}}
	+ \frac{3}{\Delta_n} \sum_{t_{j} \in (t_1,\pi]} \frac{1}{x_j^\alpha} \Delta_n
	\big)
	\\
	& \le 4\Delta_n^2 
	\big( \dfrac{1}{\Delta_n^2} \int_{0}^{\pi/2} \int_{\pi/n}^{2\pi} \dfrac{1}{r^\alpha} \, r \, dr \, d\theta
	+ \frac{3}{x_1^{\alpha}}
	+ \frac{3}{\Delta_n} \int_{\pi/n}^{2\pi} \frac{1}{t^\alpha} \, dt \big).
\end{align*}
Since
\[
	\int_{0}^{\pi/2} \int_{\pi/n}^{2\pi} \dfrac{1}{r^\alpha} \, r \, dr \, d\theta
	= \frac{\pi}{2} \times
	\left\{
		\begin{array}{ll}
			\frac{\pi^{2-\alpha}}{\alpha-2} \big( n^{\alpha-2} - 2^{2-\alpha} \big), & \alpha \ne 2, \vspace{0.1cm} \\
			\log 2n, & \alpha = 2,
		\end{array}
	\right.
\]
and
\[
	\int_{\pi/n}^{2\pi} \frac{1}{t^\alpha} \, dt
	= \left\{
		\begin{array}{ll}
			\frac{\pi^{1-\alpha}}{\alpha-1} \big( n^{\alpha-1}-2^{1-\alpha} \big), & \alpha \ne 1, \vspace{0.1cm} \\
			\log 2n, & \alpha = 1,
		\end{array}
	\right.
\]
the result follows from the above estimate when $\alpha >0$.
When $\alpha \le 0$, $\dfrac{1}{\Vert \mathbf{x} \Vert^{\alpha}}$ is bounded on $[-\pi,\pi]^2$, and therefore the result is immediate.
\end{proof}
\begin{Lem} [cf. Lemma 1 in \cite{BEY}] \label{lemma:old}
For a twice continuously differentiable function $h(\mathbf{x})$ on the domain
$R = [a-{\delta_1\over 2},a+{\delta_1\over 2}] \times [b-{\delta_2\over 2},b+{\delta_2\over 2}] \subset \mathbb{R}^2$,
we have
\begin{equation} \label{eq: 4.3}
	\Big| h(a,b) - \frac{1}{\delta_1 \, \delta_2} \iint_R h(\mathbf{x}) \, d\mathbf{x}\Big|
	\leq \frac{1}{24}
	\big( \delta_1^2 \, \displaystyle \max_{\mathbf{x} \in R} |\partial_1^2h(\mathbf{x})|
	+ \delta_2^2 \, \max_{\mathbf{x} \in R} |\partial_2^2 h(\mathbf{x})| \big).
\end{equation}
\end{Lem}

As a preparation for the proof of Theorem~\ref{thm:main1}, we note the following.

\begin{Thm} \label{thm:diiference-order}
Given a fixed $\beta \in (0,1)$ and $m \ge 1$, the difference between any two terms in
$\{ F_n(f), F_n^{\beta}(f), F_n(f_1), F_n^{\beta}(f_m), 
I_n(f), I_n^{\beta}(f), I_n(f_1), I_n^{\beta}(f_m) \}$ is $\mathcal{O}(n^2)$ as $n \to \infty$.
\end{Thm}
\begin{proof}
First note that
\[
	F_n(f) - F_n^{\beta}(f) = \sum_{\mathbf{t}_{j,k} \in D_n \backslash D_{n}^{\beta}} f(\mathbf{t}_{j,k})
	\quad
	\text{and}
	\quad
	I_n(f) - I_n^{\beta}(f) = \frac{1}{\Delta_n^2} \iint_{D_n \backslash D_{n}^{\beta}} f(\mathbf{x}) d\mathbf{x}.
\]
Since $\# \{ \mathbf{t}_{j,k} \in D_n \backslash D_{n}^{\beta} \} \le n^2$ and $f$ is bounded in
$D_n \backslash D_{n}^{\beta}$,
we therefore deduce
\[
	F_n(f) - F_n^{\beta}(f) = \mathcal{O}(n^2)
	\quad
	\text{and}
	\quad
	I_n(f) - I_n^{\beta}(f) = \mathcal{O}(n^2)
\]
as $n \to \infty$.
As for $F_n(f)-I_n(f)$, we use Lemma~\ref{lemma:old} to estimate
\begin{align*}
	|F_n(f)-I_n(f)|
	& \le \sum_{\mathbf{t}_{j,k} \in D_n} \big| f(\mathbf{t}_{j,k}) - \frac{1}{\Delta_n^2} \int_{X_{j,k}} f(\mathbf{x}) \, d\mathbf{x} \big|
	\\
	& \le \frac{1}{24} \Delta_n^2 \sum_{\mathbf{t}_{j,k} \in D_n}
	\big( \max_{\mathbf{x} \in X_{j,k}} |\partial_1^2f(\mathbf{x})| + \max_{\mathbf{x} \in X_{j,k}} |\partial_2^2f(\mathbf{x})| \big).
\end{align*}
Accordingly, Corollary~\ref{cor:est-f2} gives
\begin{align*}
	|F_n(f)-I_n(f)|
	& \le \frac{C_3( {\pi \over 2} )}{12} \Delta_n^2 \sum_{\mathbf{t}_{j,k} \in D_n} \max_{\mathbf{x} \in X_{j,k}} \frac{1}{\Vert \mathbf{x} \Vert^4}.
\end{align*}
Therefore, by Lemma~\ref{lemma:radialsum}, we have
\[
	F_n(f)-I_n(f) = \mathcal{O}(n^2),
	\qquad
	\text{as }
	n \to \infty.
\]
These show that the difference between any two terms in $\{ F_n(f), F_n^{\beta}(f), I_n(f), I_n^{\beta}(f) \}$
is $\mathcal{O}(n^2)$ as $n \to \infty$.

Next note that
\[
	F_n(f_1) - F_n^{\beta}(f_1) = \sum_{\mathbf{t}_{j,k} \in D_n \backslash D_{n}^{\beta}} f_1(\mathbf{t}_{j,k})
	\quad
	\text{and}
	\quad
	I_n(f_1) - I_n^{\beta}(f_1) = \frac{1}{\Delta_n^2} \iint_{D_n \backslash D_{n}^{\beta}} f_1(\mathbf{x}) d\mathbf{x}.
\]
Similar to above, noting that $\# \{ \mathbf{t}_{j,k} \in D_n \backslash D_{n}^{\beta} \} \le n^2$ and $f_1$ is bounded in
$D_n \backslash D_{n}^{\beta}$, we obtain
\[
	F_n(f_1) - F_n^{\beta}(f_1) = \mathcal{O}(n^2)
	\quad
	\text{and}
	\quad
	I_n(f_1) - I_n^{\beta}(f_1) = \mathcal{O}(n^2)
\]
as $n \to \infty$, and making use of Lemma~\ref{lemma:old}
\begin{align*}
	|F_n(f_1)-I_n(f_1)|
	\le \frac{1}{24} \Delta_n^2 \sum_{\mathbf{t}_{j,k} \in D_n}
	\big( \max_{\mathbf{x} \in X_{j,k}} |\partial_1^2f_1(\mathbf{x})| + \max_{\mathbf{x} \in X_{j,k}} |\partial_2^2f_1(\mathbf{x})| \big).
\end{align*}
Choosing $C_4 = 2\pi$ in the first part of Corollary~\ref{cor:fm-2}, we therefore have
\begin{align*}
	|F_n(f_1)-I_n(f_1)|
	\le \frac{C_4'}{12} \Delta_n^2 \sum_{\mathbf{t}_{j,k} \in D_n} \max_{\mathbf{x} \in X_{j,k}} \frac{1}{\Vert \mathbf{x} \Vert^4}.
\end{align*}
This, in turn, implies by Lemma~\ref{lemma:radialsum}
\[
	F_n(f_1)-I_n(f_1) = \mathcal{O}(n^2),
	\qquad
	\text{as }
	n \to \infty.
\]
Therefore the difference between any two terms in $\{ F_n(f_1), F_n^{\beta}(f_1), I_n(f_1), I_n^{\beta}(f_1) \}$
is also $\mathcal{O}(n^2)$ as $n \to \infty$.

Considering $F_n^{\beta}(f_m) - I_n^{\beta}(f_m)$, we use Lemma~\ref{lemma:old} to obtain
\begin{align*}
	|F_n^{\beta}(f_m) - I_n^{\beta}(f_m)|
	& \le \sum_{\mathbf{t}_{j,k} \in D_n^{\beta}}
	\big| f_m(\mathbf{t}_{j,k}) - \frac{1}{\Delta_n^2} \iint_{X_{j,k}} f_m(\mathbf{x}) \, d\mathbf{x} \big|
	\\
	& \le \frac{1}{24} \Delta_n^2 \sum_{\mathbf{t}_{j,k} \in D_n^{\beta}}
	\big( \max_{\mathbf{x} \in X_{j,k}} |\partial_1^2f_m(\mathbf{x})| + \max_{\mathbf{x} \in X_{j,k}} |\partial_2^2f_m(\mathbf{x})| \big)
\end{align*}
so that Corollary~\ref{cor:fm-2} entails
\begin{align*}
	|F_n^{\beta}(f_m) - I_n^{\beta}(f_m)|
	& \le \frac{C_5(\beta)}{12} \Delta_n^2 \sum_{\mathbf{t}_{j,k} \in D_n^{\beta}}
	\max_{\mathbf{x} \in X_{j,k}} \frac{1}{\Vert \mathbf{x} \Vert^4}
	\le \frac{C_5(\beta)}{12} \Delta_n^2 \sum_{\mathbf{t}_{j,k} \in D_n}
	\max_{\mathbf{x} \in X_{j,k}} \frac{1}{\Vert \mathbf{x} \Vert^4}.
\end{align*}
Thus, by Lemma~\ref{lemma:radialsum}, we have
\[
	F_n^{\beta}(f_m) - I_n^{\beta}(f_m) = \mathcal{O}(n^2),
	\qquad
	\text{as }
	n \to \infty.
\]

To complete the proof, it is therefore sufficient to show for $m \ge 1$ that
\begin{equation} \label{eq:Indif}
	I_n^{\beta}(f) - I_n^{\beta}(f_m) = \mathcal{O}(n^2),
	\qquad
	\text{as }
	n \to \infty.
\end{equation}
Indeed, 
\begin{align*}
	|I_n^{\beta}(f) - I_n^{\beta}(f_m)|
	\le \frac{1}{\Delta_n^2} \iint_{D_n^{\beta}} |f(\mathbf{x}) - f_m(\mathbf{x})| \, d\mathbf{x}
	= \frac{1}{\Delta_n^2} \iint_{D_n^{\beta}} \frac{|\psi(\mathbf{x}) - p_m(\mathbf{x})|}{\psi(\mathbf{x}) \, p_m(\mathbf{x})} \, d\mathbf{x}
\end{align*}
so that by Corollary~\ref{cor:phi-p}, and Lemmas~\ref{lemma:phi-lower} and \ref{lemma:p-upper-lower}, we have
\begin{equation} \label{eq:Indiffinal}
	|I_n^{\beta}(f) - I_n^{\beta}(f_m)|
	\le \dfrac{1}{\Delta_n^2} \, \dfrac{2|\Phi|}{\delta \, C_1({\pi \over 2})} \, \dfrac{\overline{s}^{2m+2}}{(2m+2)!} \, \iint_{D_n^{\beta}} \Vert \mathbf{x} \Vert^{2m-2} \, d\mathbf{x}.
\end{equation}
Since the integral on the right-hand side of \eqref{eq:Indiffinal} is finite, \eqref{eq:Indif} follows and this completes the proof.
\end{proof}

Theorem~\ref{thm:diiference-order} allows us to complete the proof of Theorem~\ref{thm:main1} provided we calculate
one of the terms in $\{ F_n(f), F_n^{\beta}(f), F_n(f_1), F_n^{\beta}(f_m),I_n(f), I_n^{\beta}(f), I_n(f_1), I_n^{\beta}(f_m) \}$.
To this end, we calculate $I_n(f_1)$. The result of this calculation is recorded as Theorem~\ref{thm:Inf1} in \S\ref{sec:mainresults} which,
in turn, completes the proof of Theorem~\ref{thm:main1}. In connection with the proof Theorem~\ref{thm:Inf1}, let us note that, with $a,b$, and $c$ as in \eqref{eq:abc},
we have
\begin{equation} \label{eq:f_1}
	f_1(x,y) = \dfrac{1}{p_1(x,y)} = \dfrac{2 \, |\Phi|}{ax^2+bxy+cy^2}.
\end{equation}
Note that \eqref{eq:s1-s2} implies $a,c>0$, and since $\mathbf{s}_1 \scriptstyle{\backslash}$$\!\!\!\sslash \mathbf{s}_2$ 
Cauchy-Schwarz inequality shows that (cf. \eqref{eq:S})
\[
	d = -4\det(S^TS) < 0.
\]

\vspace{0.1cm}

\noindent
\emph{Proof of Theorem 5.}
By \eqref{eq:f_1}
\[
	I_n(f_1)
	= \frac{1}{\Delta_n^2} \iint_{D_n} f_1 \, dx \, dy
	= \frac{1}{\Delta_n^2} \iint_{D_n} \frac{2|\Phi|}{ax^2 + b xy + cy^2} dx \, dy
\]
where $a,b,c$ are as in \eqref{eq:abc}. Since $f_1(x,y)=f_1(y,x)$, setting
\begin{align*}
E(R)  & := \Big( [0,R] \times [-R,R] \Big) \backslash \Big( [0,\frac{\pi}{n}] \times [-\frac{\pi}{n},\frac{\pi}{n}] \Big),\\
E_n & := [\pi - {\pi \over n},\pi + {\pi \over n}] \times [-\pi - {\pi \over n},-\pi + {\pi \over n}] 
\end{align*}
and making use of the explicit form of $D_n$ in \eqref{eq:D_n} we therefore have
\begin{equation} \label{eq:I_n^1}
	I_n(f_1) = \frac{1}{\Delta_n^2} 
	\left\{
		\begin{array}{ll}
			2 \iint_{E(\pi)} f_1 \, dx \, dy, & n \text{ is odd},
			\vspace{0.1cm} \\
			\left( \iint_{E(\pi+\frac{\pi}{n})} + \iint_{E(\pi-\frac{\pi}{n})} - \iint_{E_n} \right) f_1 \, dx \, dy, & n \text{ is even}.
		\end{array}
	\right.
\end{equation}
Indeed, for $R > \frac{\pi}{n}$, integrating in polar coordinates, we have
\begin{align}
	\iint_{E(R)} f_1 \, dx \, dy
	& = \int_{-\frac{\pi}{2}}^{-\frac{\pi}{4}} \int_{-\frac{\pi}{n \, \sin \theta}}^{-\frac{R}{\sin \theta}}
	\frac{2|\Phi|}{a\cos^2\theta + b \cos\theta \sin \theta + c \sin^2 \theta} \frac{1}{r} dr d\theta
	\label{eq:intf1}
	\\
	& + \int_{-\frac{\pi}{4}}^{\frac{\pi}{4}} \int_{\frac{\pi}{n \, \cos \theta}}^{\frac{R}{\cos \theta}}
	\frac{2|\Phi|}{a\cos^2\theta + b \cos\theta \sin \theta + c \sin^2 \theta} \frac{1}{r} dr d\theta
	\nonumber
	\\
	& + \int_{\frac{\pi}{4}}^{\frac{\pi}{2}} \int_{\frac{\pi}{n \, \sin \theta}}^{\frac{R}{\sin \theta}}
	\frac{2|\Phi|}{a\cos^2\theta + b \cos\theta \sin \theta + c \sin^2 \theta} \frac{1}{r} dr d\theta
	\nonumber
	\\
	& = \log \Big( \frac{n \, R}{\pi} \Big)
	\int_{-\frac{\pi}{2}}^{\frac{\pi}{2}}
	\frac{2|\Phi|}{a\cos^2\theta + b \cos\theta \sin \theta + c \sin^2 \theta} d\theta
	\nonumber
	\\
	& = \log \Big( \frac{n \, R}{\pi} \Big)
	\frac{4|\Phi|}{\sqrt{|d|}} \arctan \Big( \frac{b+2c \, \tan \theta}{\sqrt{|d|}} \Big) \Big|_{-\frac{\pi}{2}}^{\frac{\pi}{2}}
	\nonumber
	\\
	& = \log \Big( \frac{n \, R}{\pi} \Big)
	\frac{4|\Phi|\pi}{\sqrt{|d|}}
	= \log \Big( \frac{n \, R}{\pi} \Big)
	\frac{2|\Phi|\pi}{\sqrt{\det(S^TS)}}.
	\nonumber
\end{align}
Using \eqref{eq:intf1} in \eqref{eq:I_n^1}, we obtain
\[
	I_n(f_1) = \frac{|\Phi|}{\pi \sqrt{\det(S^TS)}} \, n^2 \, \log n
\]
when $n$ is odd, and
\[
	I_n(f_1) = \frac{|\Phi|}{\pi \sqrt{\det(S^TS)}} \, n^2 \, \big( \log n + \frac{1}{2} \log\big( 1 - \frac{1}{n^2} \big) \big)
	- {1 \over \Delta_n^2} \iint_{E_n} f_1 \, dx \, dy
\]
when $n$ is even. As for the very last integral, since $\partial_1^2 f_1, \partial_2^2 f_1 \asymp 1$ on $E_n$,
Lemma~\ref{lemma:old} gives
\[
	{1 \over \Delta_n^2} \iint_{E_n} f_1 \, dx \, dy
	= f_1(-\pi,\pi) + \mathcal{O} \big( {1 \over n^2} \big) 
	= {2 |\Phi| \over \pi^2} \, {1 \over a-b+c} + \mathcal{O} \big( {1 \over n^2} \big).
\]
The proof of Theorem~\ref{thm:Inf1} follows.
\hspace*{\fill} $\square$

\vspace{0.2cm}

\noindent
\emph{Proof of Theorem 2.}
First we show that
\begin{equation} \label{eq:first-est}
	F_n^{\beta}(f) - I_n^{\beta}(f) + I_n^{\beta}(f_m) - F_n^{\beta}(f_m) = 
	\left\{ 	
		\begin{array}{cl}
			\mathcal{O}(\log n), & m = 1,
			\\
			\mathcal{O}(1), & m \ge 2,	
		\end{array}
	\right.
\end{equation}
as $n \to \infty$. Indeed, with $h_m = f-f_m$, we have
\[
	F_n^{\beta}(f) - I_n^{\beta}(f) + I_n^{\beta}(f_m) - F_n^{\beta}(f_m)
	= \sum_{\mathbf{t}_{j,k} \in D_{n}^{\beta}} \Big( h_m(\mathbf{t}_{j,k}) - \frac{1}{\Delta_n^2} \int_{X_{j,k}} h_m(\mathbf{x}) \, d\mathbf{x} \Big)
\]
so that an appeal to Lemma~\ref{lemma:old} gives
\[
	|F_n^{\beta}(f) - I_n^{\beta}(f) + I_n^{\beta}(f_m) - F_n^{\beta}(f_m)|
	\le \frac{\Delta_n^2}{24} \sum_{\mathbf{t}_{j,k} \in D_{n}^{\beta}}
	\big( \max_{\mathbf{x} \in X_{j,k}} |\partial_1^2 h_m(\mathbf{x})| + \max_{\mathbf{x} \in X_{j,k}} |\partial_2^2 h_m(\mathbf{x})| \big).
\]
Lemma \ref{lemma:der-est}, in turn, implies that
\[
	|F_n^{\beta}(f) - I_n^{\beta}(f) + I_n^{\beta}(f_m) - F_n^{\beta}(f_m)|
	\le \frac{C_6(\beta)}{12} \, \Delta_n^2 \, 
	\sum_{\mathbf{t}_{j,k} \in D_{n}^{\beta}} \max_{\mathbf{x} \in X_{j,k}} \Vert \mathbf{x} \Vert^{2m-4}
\]
so that
\[
	|F_n^{\beta}(f) - I_n^{\beta}(f) + I_n^{\beta}(f_m) - F_n^{\beta}(f_m)|
	\le \frac{C_6(\beta)}{12}  \, \Delta_n^2 \sum_{\mathbf{t}_{j,k} \in D_{n}} \max_{\mathbf{x} \in X_{j,k}} \Vert \mathbf{x} \Vert^{2m-4}.
\]
Thus \eqref{eq:first-est} follows from Lemma~\ref{lemma:radialsum}.
To complete the proof, it now suffices to show that
\[
	F_n(f)-I_n(f)+I_n^{\beta}(f)-F_n^{\beta}(f)
	= \mathcal{O}(1),
	\qquad
	\text{as } n \to \infty. 
\]
Arguing as above, we have
\[
	|F_n(f)-I_n(f)+I_n^{\beta}(f)-F_n^{\beta}(f)|
	\le \frac{\Delta_n^2}{24} \sum_{\mathbf{t}_{j,k} \in D_{n} \backslash D_{n}^{\beta}}
	\big( \max_{\mathbf{x} \in X_{j,k}} |\partial_1^2 f(\mathbf{x})| + \max_{\mathbf{x} \in X_{j,k}} |\partial_2^2 f(\mathbf{x})| \big)
\]
whose right-hand side is $\mathcal{O}(1)$ as $n \to \infty$ since $|\partial_k^2 f(\mathbf{x})|$ is bounded on
$D_{n} \backslash D_{n}^{\beta}$ and $\# \{ \mathbf{t}_{j,k} \in D_n \backslash D_{n}^{\beta} \} \le n^2$. This finishes the proof.
\hspace*{\fill} $\square$

\section{The Sum $G_n$} \label{sec:Gn}

Let $aj^2+bjk+ck^2$ be a positive definite binary quadratic form of discriminant $d$, so that
\begin{equation}
	d = b^2 - 4ac < 0, \quad 0 < a \leq c, \quad b \ge 0,
\label{eq: 1.1a}
\end{equation}
(the last inequalities are taken to hold by the symmetries in the problem).
Since $d<0$, we have $b < 2\sqrt{ac} \le 2c$, and $b^2 +|d| =4ac$.

We first re-express $G_n$ of \eqref{eq:G_n} as
\begin{equation*}
G_n = \dfrac{i}{\sqrt{|d|}}\sum_{j,k = 1}^{n-1} \dfrac{1}{j} 
\Big( \frac{1}{k+\mu j}-\frac{1}{k+\overline{\mu}j}-\frac{1}{k-\mu j}+\frac{1}{k-\overline{\mu}j}\Big)
%\label{eq: 1.2}
\end{equation*}
where $\mu$ is given by \eqref{eq:mu}.
%\begin{equation*}
%\mu = \frac{-b+\sqrt{|d|}i}{2c}.
%%\label{eq: 1.3}
%\end{equation*}
As (see 6.3.6 of \cite{AS})
\begin{equation*}
\sum_{k=1}^{n-1} \frac{1}{k+z} = \psi(n+z)-\psi(1+z),\qquad (-z \notin \mathbb{N}),
%\label{eq: 1.4}
\end{equation*}
we see that
\begin{align*}
G_n  = \dfrac{i}{\sqrt{|d|}}\sum_{j = 1}^{n-1}\frac{1}{j}
\Big( \psi &(n+\mu j)  - \psi(1+\mu j) - \psi (n+\overline{\mu}j) + \psi (1+\overline{\mu}j) \nonumber \\
	&  - \psi(n-\mu j) + \psi(1-\mu j) + \psi (n-\overline{\mu}j) - \psi (1-\overline{\mu}j) \Big).
%\label{eq: 1.5}
\end{align*}
Now we use the asymptotic formula
\begin{align*}
	\psi(z) = \log z -\frac{1}{2z} - \frac{1}{12z^2}+ \mathcal{O} \big( \frac{1}{|z|^4}\big)
	\quad
	\text{as } |z| \to \infty, \
	|\arg(z)| < \pi-\beta,
%\label{eq: 1.6}
\end{align*}
(see \cite{Lg} for a discussion of this asymptotic) for the $\psi(n + \cdots)$ terms, 
the recurrence and reflection formulas
\begin{align*} 
\psi(1+z) = \psi(z) + \frac{1}{z}, \quad \psi(1-z) = \psi(z) + \pi \cot \pi z,
%\label{eq: 1.7}
\end{align*}
for the $\psi(1 + \cdots)$ terms, and $\displaystyle
\cot z = -i \big( 1+ {2e^{2iz} \over 1- e^{2iz}} \big) = i \big( 1+ {2e^{-2iz} \over 1- e^{-2iz}} \big) $
to write
\begin{align}
	G_n
	& = \dfrac{i}{\sqrt{|d|}}
	\sum_{j = 1}^{n-1}
	\frac{1}{j}
	\Big( \log(n+\mu j) - \log(n+\overline{\mu} j) - \log(n-\mu j) + \log(n-\overline{\mu} j)  \nonumber
	\\
	& \hspace{2.4cm}
	- \frac{1}{2} \Big( \frac{1}{n+\mu j} - \frac{1}{n+\overline{\mu} j} - \frac{1}{n-\mu j} + \frac{1}{n-\overline{\mu} j} \Big) \nonumber
	\\
	& \hspace{2.4cm}
	- \frac{1}{12} \Big( \frac{1}{(n+\mu j)^2} - \frac{1}{(n+\overline{\mu} j)^2}
	- \frac{1}{(n-\mu j)^2} + \frac{1}{(n-\overline{\mu} j)^2} \Big) \nonumber 
	\\
	& \hspace{2.4cm}
	+ \pi \cot (\pi \mu j) - \pi \cot (\pi \overline{\mu} j)
	- \frac{1}{\mu j} + \frac{1}{\overline{\mu} j}
	 \Big) 
	+ \mathcal{O} \Big( {\log n \over n^4} \Big) \nonumber
	\\
	& = \dfrac{i}{\sqrt{|d|}}
	\sum_{j = 1}^{n-1}
	\frac{1}{j}
	\Big( 2i \arctan \frac{\frac{\sqrt{|d|}j}{2c}}{n-\frac{bj}{2c}} + 2i \arctan \frac{\frac{\sqrt{|d|}j}{2c}}{n+\frac{bj}{2c}} %\nonumber
%	\\
%	& \hspace{2.4cm}	
	- \frac{1}{2} \Big( \frac{(\overline{\mu}-\mu)j}{|n+\mu j|^2} + \frac{(\overline{\mu}-\mu)j}{|n-\mu j|^2} \Big) \nonumber
	\\
	& \hspace{2.4cm}	
	- \frac{1}{12} \Big( \frac{2(\overline{\mu}-\mu)nj+(\overline{\mu}^2-\mu^2)j^2}{|n+\mu j|^4}
				     + \frac{2(\overline{\mu}-\mu)nj-(\overline{\mu}^2-\mu^2)j^2}{|n-\mu j|^4} \Big) \nonumber
	\\
	& \hspace{2.4cm}
	-2i\pi \Big( {e^{2i\pi\mu j} \over 1-e^{2i\pi\mu j}} + {e^{-2i\pi\overline{\mu} j} \over 1-e^{-2i\pi\overline{\mu} j}} + 1 \Big)
	+ \big( \frac{1}{\overline{\mu}} - \frac{1}{\mu} \big) \frac{1}{j}
	\Big)
	+ \mathcal{O} \Big( {\log n \over n^4} \Big) \nonumber
	\\
	& = -\dfrac{2}{\sqrt{|d|}}
	\sum_{j = 1}^{n-1}
	\frac{1}{n} \frac{1}{\frac{j}{n}}
	\Big( \arctan \frac{\sqrt{|d|}\frac{j}{n}}{2c-b\frac{j}{n}}
	+ \arctan \frac{\sqrt{|d|}\frac{j}{n}}{2c+b\frac{j}{n}} \Big) \nonumber
	\\
	& - \dfrac{1}{2n}
	\sum_{j = 1}^{n-1}
	\frac{1}{n}
	\Big(
		\frac{1}{a\big(\frac{j}{n}\big)^2-b \frac{j}{n}+c}
		+ \frac{1}{a\big(\frac{j}{n}\big)^2+b \frac{j}{n}+c} \nonumber
	\Big)
	\\
	& - \frac{1}{12 \, n^2} \sum_{j = 1}^{n-1} \frac{1}{n}
	\Big( \frac{2c-b{j \over n}}{(a({j \over n})^2-b{j \over n}+c)^2}
	+ \frac{2c+b{j \over n}}{(a({j \over n})^2+b{j \over n}+c)^2} \Big) \nonumber
	\\
	& +{4\pi \over \sqrt{|d|}} \Re \Big( \sum_{j = 1}^{n-1} \frac{1}{j} {e^{2i\pi\mu j} \over 1-e^{2i\pi\mu j}} \Big)
	+{2\pi \over \sqrt{|d|}} \sum_{j = 1}^{n-1} \frac{1}{j}
	- \dfrac{1}{a}
	\sum_{j = 1}^{n-1}
	\frac{1}{j^2}
	+ \mathcal{O} \Big( {\log n \over n^4} \Big) \nonumber
	\\
	& = -\dfrac{2}{\sqrt{|d|}} \mathcal{G}_1
	- \dfrac{1}{2n} \mathcal{G}_2
	- \frac{1}{12 \, n^2} \mathcal{G}_3
	+{4\pi \over \sqrt{|d|}} \mathcal{G}_4
	+{2\pi \over \sqrt{|d|}} \mathcal{G}_5
	- \dfrac{1}{a} \mathcal{G}_6 \nonumber
	+ \mathcal{O} \Big( {\log n \over n^4} \Big),
	\quad
	\text{say}. \nonumber
	\\
\label{eq: 1.8} \end{align}
For $\mathcal{G}_1,\, \mathcal{G}_2$, and $\mathcal{G}_3$, we use the following version of the Euler-Maclaurin
summation formula (see Lampret \cite{La}): For any integers $n,\, p \geq 1$ and any function $g\in C^{p}[0,1]$,
\begin{align}
\sum_{j=1}^{n-1} \frac{1}{n}  g \big( \frac{j}{n} \big) 
= \int_{0}^{1}g(x)\, dx - {1 \over n}g(0) + \sum_{\ell=1}^{p} \frac{1}{n^{\ell}} {B_{\ell}\over {\ell}!} \big[g^{({\ell}-1)}(x)\big]_{0}^{1} + r_{p,n}(g),
\label{eq: 1.9} \end{align}
where $B_{\ell}$ are the Bernoulli numbers ($B_1 = -{1 \over 2}$, $B_2 = {1 \over 6}$, $B_3=0$, $B_4=-{1\over 30}$), and
\begin{align*}
r_{p,n}(g) = -\frac{1}{n^p}\frac{1}{p!} \int_{0}^{1} B_{p} (-nx) \, g^{(p)}(x)\, dx
%\label{eq: 1.10}
\end{align*}
with $B_{p}(x)$ being the $p$-th Bernoulli polynomial in $[0,1)$ extended to all  $x \in \mathbb{R}$ via
$B_{p}(x+1)=B_{p}(x)$.
%%%%%%%%%%%%%%%%%%%%%%%%%%%%%%%%%
%For $G_1$, expressing the contributions
%of each of the two terms of the summand as
%\begin{align}
%-{\sqrt{|d|}\over 2an}+ \sum_{j=0}^{n-1}{1\over n}\Big({1\over {j\over n}}\arctan\big({\sqrt{|d|}{j\over n}\over 2c\pm b({j\over n})}\big)\Big),
%\label{eq: 1.11}\end{align}
%we take 
% $\displaystyle f(x) = f_{\pm}(x) ={1\over x}\arctan\big({\sqrt{|d|}x\over 2c \pm bx}\big)$ for the very last sum. 
%By employing limits as $x\to 0^+$ and $x\to 1^-$ for removable singularities (and that is how we have included the $k=0$ term in the sum in (1.11)), we have 
%$\displaystyle  f_{\pm}(1) = \arctan{\sqrt{|d|}\over 2c\pm b},
%f_{\pm}'(1) = {\sqrt{|d|}\over 2(a \pm b +c)}- \arctan{\sqrt{|d|}\over 2c \pm b},
%f_{\pm}(0) = {\sqrt{|d|}\over 2c},  f_{\pm}'(0) = \mp {\sqrt{|d|}b\over 4c^2} $, 
%Since $f_{\pm}^{(\mathrm{iv})}(x)$ remains bounded on $[0,1]$, we see that $r_{4,n}(f)  \ll n^{-4}$, and (1.9) with $p=4$  now reads
%%%%%%%%%%%%%%%%%%%%%%%%%%%%%%%%%%

For $\mathcal{G}_{1}$, we apply \eqref{eq: 1.9} to $g_1 = g_{-}+g_{+}$ with
\[
	g_{\pm}(x) ={1\over x}\arctan\big({\sqrt{|d|}x\over 2c \pm bx}\big).
\]
Integration by parts followed by partial fraction expansion gives
\begin{align*}
	\int_{0}^{1}g_1(x) dx
	& = -{\sqrt{|d|}\over 2}\int_{0}^{1} \Big( {\log x\over ax^2 - bx + c} + {\log x\over ax^2 + bx + c} \Big) \, dx
	\\
	& =  \Im \int_{0}^{1}(\log x)\Big({1\over x + { b+\sqrt{|d|}i\over 2a}} + {1\over x - { b-\sqrt{|d|}i\over 2a}}\Big)\, dx.
%\label{eq: 1.13}
\end{align*}
Observe that
\begin{align*}
\int_{0}^{1} {\log x\over x+w} \, dx & ={1\over w} \int_{0}^{1} {\log x\over 1+{x\over w}} \, dx=(\log x)\log(1+{x \over w})\Big|_{0}^{1}
-\int_{0}^{1} {\log(1+{x \over w})\over x}\, dx \nonumber \\
& =-\int_0^1 {\log(1+{x \over w})\over x}\, dx=-\int_0^{-{1\over w}} {\log(1-u)\over u}\, du = \mathrm{Li}_{2}\big(-{1\over w}\big),
%\label{eq: 1.15}
\end{align*}
where (see \cite{Ki}, \cite {Z})
\begin{align*}
\mathrm{Li}_{2}(z)  := & \sum_{m=1}^{\infty}{z^m\over m^2}  \quad (|z|\leq 1) 
%\label{eq: 1.16}
\\
=& -\int_{0}^{z}{\log(1-u)\over u}\, du  \quad \big(z\in \mathbb{C} \setminus [1,\infty),\, \arg(1-z) \in (-\pi,\pi) \big) 
%\label{eq: 1.17}
\end{align*} 
(the integral representation is valid in the unit disk and provides an analytic continuation to the cut plane).
In our case $\displaystyle w = {\pm b +\sqrt{|d|}i\over 2a},\, |w| = \sqrt{{c\over a}} \geq 1, \, -{1\over w} = {\mp b +\sqrt{|d|}i\over 2c},\,
\Im\big(-{1\over w}\big)\neq 0$, and
we arrive at
\begin{equation*}
	\int_{0}^{1}g_1(x) \, dx
	= \Im \big[ \mathrm{Li}_{2}\big({b+\sqrt{|d|}i\over 2c}\big) + \mathrm{Li}_{2}\big({-b+\sqrt{|d|}i\over 2c} \big)\big]
	= \Im\big[\mathrm{Li}_{2}(\mu) -\mathrm{Li}_{2}(-\mu)\big].
%\label{eq: 1.18}
\end{equation*}
For an expression in terms of a real function we resort to
Kummer's formula
\begin{align*}
\text{Li}_{2}(re^{i\theta})  = & \int_{0}^{r}{\log(1-2x\cos\theta +x^2)\over 2x}\, dx % \nonumber
\\
& \quad +  i\big[ \omega\log r  + {1\over 2}\big(\text{Cl}_{2}(2\theta) + \text{Cl}_{2}(2\omega) - 
\text{Cl}_{2}(2\theta +2\omega)\big)\big]  % \nonumber
\end{align*}
for $0<r \leq 1$ with $\displaystyle \omega  = \arctan\big({r\sin\theta\over 1-r\cos\theta}\big)$
%\begin{align}
%\text{Li}_{2}(re^{i\theta})  = & \int_{0}^{r}{\log(1-2x\cos\theta +x^2)\over 2x}\, dx \nonumber \\
%& \quad +  i\big[ \omega\log r  + {1\over 2}\big(\text{Cl}_{2}(2\theta) + \text{Cl}_{2}(2\omega) - 
%\text{Cl}_{2}(2\theta +2\omega)\big)\big]  \nonumber \\
%& \mathrm{for} \;\;  0<r \leq 1 \; \mathrm{with} \;\;  \omega  = \arctan\big({r\sin\theta\over 1-r\cos\theta}\big)
%\label{eq: 1.19} \end{align}
(see (5.2)-(5.5) in \cite{L} or p. 8 of \cite{Ki}), involving Clausen's function
\begin{align*}
\text{Cl}_{2}(\theta) := \Im\big[\text{Li}_2(e^{i\theta})\big] = \sum_{k=1}^{\infty}{\sin k\theta\over k^2}.
%\label{eq: 1.20}
\end{align*}
(Kummer's formula is usually stated for $0<r<1$, but it remains valid for
$r=1$ with appropriate usage of $\omega$ values about the points $\pm 1$).
%For points on the unit circle other than $\pm 1$,
%the imaginary part of the formula can be verified using the duplication formula for Clausen's function. 
%near the point $1$, as
%$r \to 1^{-}$ the limit exists since $\omega$ is bounded, and for
%$r=1$ and $\theta\to 0^{\pm}$,  $\omega \to \pm{\pi\over 2}^{\mp}$ ).
Since
\begin{align*}
	\pm\mu = r e^{i\theta_{\pm}}
	\text{ with }
	r = \sqrt{{a \over c}},
	\
	\theta_{\pm} = \arctan \big( {b \over \sqrt{|d|}} \big) \pm {\pi \over 2},
	\
	\omega_{\pm} = \pm \Big( {\pi \over 2} - \arctan {2c \pm b \over \sqrt{|d|}} \Big),
%\label{eq: 1.21}
\end{align*}
using elementary properties of the $\arctan$ function and $\mathrm{Cl}_2(-\theta) = -\mathrm{Cl}_2(\theta)$, we get
\begin{align*}
	& \int_{0}^{1}g_1(x) \, dx
	= \Big({\pi\over 2} - \arctan\big({c-a\over \sqrt{|d|}}\big)\Big)\log\sqrt{{a\over c}}
	+ {1\over 2}\Big[\mathrm{Cl}_2\Big(\pi -2\arctan\big({2c+b\over \sqrt{|d|}}\big)\Big) %\nonumber
	\\
	& \hspace{0.7cm}+ \mathrm{Cl}_2\Big(\pi - 2\arctan\big({2c-b\over \sqrt{|d|}}\big) \Big) 
	+ \mathrm{Cl}_2\Big(2\arctan\big({2c+b\over \sqrt{|d|}}\big) - 2\arctan\big({b\over \sqrt{|d|}}\big)\Big) %\nonumber
	\\
	& \hspace*{5cm} +\mathrm{Cl}_2\Big(2\arctan\big({2c-b\over \sqrt{|d|}}\big)+ 2\arctan\big({b\over \sqrt{|d|}}\big) \Big)\Big] %\nonumber
	\\
	& \hspace*{1.8cm} = \Big({\pi\over 2} - \arctan\big({c-a\over \sqrt{|d|}}\big)\Big)\log\sqrt{{a\over c}} + {1 \over 2} C(a,b,c)
	%\label{eq: 1.22}
\end{align*}
where $C(a,b,c)$ is as given in \eqref{eq:Cabc}. Further, $\displaystyle  g_{\pm}(1) = \arctan{\sqrt{|d|}\over 2c\pm b}$,
$\displaystyle g_{\pm}'(1) = {\sqrt{|d|}\over 2(a \pm b +c)}- \arctan{\sqrt{|d|}\over 2c \pm b}$ and, by
employing limits as $x\to 0^+$ for removable singularities, we have 
$\displaystyle g_{\pm}(0) = {\sqrt{|d|}\over 2c}$, $\displaystyle g_{\pm}'(0) = \mp {\sqrt{|d|}b\over 4c^2}$. 
Since $g_1^{(\mathrm{iv})}(x)$ remains bounded on $[0,1]$, we see that
$r_{4,n}(g_1)  \ll n^{-4}$, and therefore with $p=4$ in \eqref{eq: 1.9} we have
%\begin{samepage}
\begin{align} 
	\mathcal{G}_{1}
	& =  \Big({\pi\over 2} - \arctan\big({c-a\over \sqrt{|d|}}\big)\Big)\log\sqrt{{a\over c}}
	+ {1\over 2} C(a,b,c)
	- {1 \over 2n} \Big( {\sqrt{|d|} \over c} + {\pi \over 2} - \arctan {c-a \over \sqrt{|d|}} \Big)
	\nonumber
	\\
	& + {1 \over 12 n^2}
	\Big( {\sqrt{|d|} \over 2} \big( {1 \over a-b+c} + {1 \over a+b+c} \big) - {\pi \over 2} + \arctan{c-a \over \sqrt{|d|}} \Big)
	+ \mathcal{O} \Big( {1 \over n^4} \Big).
\label{eq: 1.23}\end{align}
%\end{samepage}

For $\mathcal{G}_{2}$, the associated function is
\[
	g_2(x) = \frac{1}{ax^2-bx+c} + \frac{1}{ax^2+bx+c},
\]
and we have
\begin{align*}
	\int_0^1 g_2(x) \, dx
	& = {2 \over \sqrt{|d|}} \Big( \arctan \frac{2ax-b}{\sqrt{|d|}} + \arctan \frac{2ax+b}{\sqrt{|d|}} \Big)\Big|_0^1 %\nonumber
	= {2 \over \sqrt{|d|}} \Big( {\pi \over 2} - \arctan {c-a \over \sqrt{|d|}} \Big),
%\label{eq: 1.24}
\end{align*}
and
$\displaystyle g_2(0) = {2 \over c},\, g_2'(0)=0,\, g_2(1) = {1 \over a-b+c} + {1 \over a+b+c},\,
g_2'(1) = - {2a-b \over (a-b+c)^2} - {2a+b \over (a+b+c)^2}$. Therefore
\begin{align} 
	\mathcal{G}_{2}
	& = {2 \over \sqrt{|d|}} \Big( {\pi \over 2} - \arctan {c-a \over \sqrt{|d|}}\Big)
	- {1 \over 2n} \Big( {2 \over c} + {1 \over a-b+c} + {1 \over a+b+c} \Big) 	\nonumber \\
	& \qquad - {1 \over 12n^2} \Big( {2a-b \over (a-b+c)^2} + {2a+b \over (a+b+c)^2} \Big) 
	+ \mathcal{O} \Big( {1 \over n^4} \Big).
\label{eq: 1.25}\end{align}

For $\mathcal{G}_{3}$, the corresponding function
\[
	g_3(x) = \frac{2c-bx}{(ax^2-bx+c)^2} + \frac{2c+bx}{(ax^2+bx+c)^2}
\]
has
\begin{align*}
	\int_0^1 g_3(x) \, dx
	& = \Big[ \frac{x}{ax^2-bx+c} + \frac{x}{ax^2+bx+c}
	+ {2 \over \sqrt{|d|}} \Big( \arctan \frac{2ax-b}{\sqrt{|d|}} + \arctan \frac{2ax+b}{\sqrt{|d|}} \Big) \Big]_0^1
	% \label{eq: 1.26}
	\\
	& = {1 \over a-b+c} + {1 \over a+b+c} + {2 \over \sqrt{|d|}} \Big( {\pi \over 2} - \arctan{c-a \over \sqrt{|d|}} \Big) %\nonumber
\end{align*}
and $\displaystyle g_3(0) = {4 \over c},\, g_3(1) = \frac{2c-b}{(a-b+c)^2} + \frac{2c+b}{(a+b+c)^2}$,
and therefore
\begin{align} 
	\mathcal{G}_{3}
	 = & {1 \over a-b+c} + {1 \over a+b+c} + {2 \over \sqrt{|d|}} \Big( {\pi \over 2} - \arctan{c-a \over \sqrt{|d|}} \Big) \nonumber \\
& \quad 	- {1 \over 2n} \Big( {4 \over c} + \frac{2c-b}{(a-b+c)^2} + \frac{2c+b}{(a+b+c)^2} \Big)
	+ \mathcal{O} \Big( {1 \over n^2} \Big).
\label{eq: 1.27}\end{align}

In  $\mathcal{G}_{4}$, let $ \displaystyle q := e^{2i\pi\mu}$ (then $\displaystyle |q| = e^{-{\pi \sqrt{|d|} \over c}}$).
By the alternating series test 
\begin{align*}
	\sum_{j = 1}^{n-1} \frac{1}{j} {e^{2i\pi\mu j} \over 1-e^{2i\pi\mu j}}
	& = \sum_{j = 1}^{n-1} \frac{1}{j} {q^{j} \over 1-q^{j}}
	= \sum_{j=1}^{\infty} \frac{1}{j} \dfrac{q^{j}}{1 - q^{j}} + O\big( |q|^n \big),
%\label{eq: 1.28}
\end{align*}
and
\begin{align*}
	\sum_{j=1}^{\infty} \frac{1}{j} \dfrac{q^j}{1-q^j} 
	& = \sum_{j=1}^{\infty} \sum_{k = 0}^{\infty} \frac{1}{j}  (q^{k+ 1})^j
	= - \sum_{k = 0}^{\infty} \sum_{j=1}^{\infty} \frac{-1}{j}  (q^{k + 1})^j
	= - \sum_{k = 0}^{\infty} \log (1 - q^{k + 1}) \nonumber	\\
	& = - \log \Big( \prod_{k=0}^{\infty} (1 - q^{k + 1}) \Big) = - \log \left( q^{-{1\over 24}}\eta(\mu) \right),
%\label{eq: 1.29}
\end{align*}
where $\eta$ is the Dedekind eta-function (see \cite{Ap}), so that
\begin{align}
	\mathcal{G}_{4} = - \Re \Big( \log \left( q^{-{1\over 24}}\eta(\mu) \right) \Big) + \mathcal{O} \big( |q|^n \big)
	= - \Big( {\pi\sqrt{|d|} \over 24 c} + \log |\eta(\mu)|  \Big) + \mathcal{O} \big( |q|^n \big).
\label{eq: 1.30}\end{align}

Next, we have (see p. 560 of \cite{Lg})
\begin{equation} \label{eq: 1.31}
	\mathcal{G}_5 = \log n +\gamma - {1\over 2n} - {1\over 12 n^2} + \mathcal{O}\big({1\over n^4}\big),
\end{equation}
and (see (6.4.3) and (6.4.12) of \cite {AS})
\begin{equation} \label{eq: 1.32}
	\mathcal{G}_6 = {\pi^2\over 6} -{1\over n} -{1\over 2n^2} -{1\over 6n^3} + \mathcal{O}\big({1\over n^5}\big).
\end{equation}

Combining the results (\ref{eq: 1.23}), (\ref{eq: 1.25}), (\ref{eq: 1.27}), (\ref{eq: 1.30}), (\ref{eq: 1.31}) and (\ref{eq: 1.32}) in (\ref{eq: 1.8}), we obtain \eqref{eq:Gnasymp} for the parameters $a,b,c$ satisfying the inequalities \eqref{eq: 1.1a}.

\begin{Rem}
It is possible to approach the calculation of $G_n$ by passing to the Epstein zeta-function associated to our binary quadratic form and then 
using using Kronecker's first limit formula. But then, even the $\mathcal{O}(1)$ part of $G_n$ seems to contain a problematic term. The
method followed here gives the lower order terms of $G_n$ as precisely as wished.  
\end{Rem}

\section{The sum $F_n(f_1)$} \label{sec:Fnf1}

With $a,b,c$ as in \eqref{eq:abc}, we have (see \eqref{eq:f_1})
\begin{align*}
	F_n(f_1)
	& = \sum_{\mathbf{t}_{j,k} \in D_n} f_1(\mathbf{t}_{j,k})
	% = \sum_{\mathbf{t}_{j,k} \in D_n} {1 \over p_1(\mathbf{t}_{j,k})}
	%\\
	% & = \sum_{\mathbf{t}_{j,k} \in D_n}
	% \dfrac{1}{\dfrac{a\big(\dfrac{2\pi j}{n}\big)^2+b\big(\dfrac{2\pi j}{n}\big)\big(\dfrac{2\pi k}{n}\big)+c\big(\dfrac{2\pi k}{n}\big)^2}
	%{2|\Phi|}}
	= \dfrac{|\Phi|n^2}{2\pi^2} \sum_{\mathbf{t}_{j,k} \in D_n} \dfrac{1}{aj^2+bjk+ck^2}
	\nonumber
	\\
	& = \dfrac{|\Phi|n^2}{2\pi^2} \times
	\left\{
		\begin{array}{cl}
			\sum\limits_{-{ n-1 \over 2} \le j,k \le {n-1 \over 2} \atop (j,k) \ne (0,0)} \dfrac{1}{aj^2+bjk+ck^2},
			& n \text{ odd},
			\\
			\sum\limits_{-{ n \over 2} \le j,k \le {n-2 \over 2} \atop (j,k) \ne (0,0)} \dfrac{1}{aj^2+bjk+ck^2},
			& n \text{ even},
		\end{array}
	\right.
	\nonumber
\end{align*}
which can be re-expressed as
\begin{align} \label{eq:Fnf1concrete}
	F_n(f_1)
	& = 	\left\{
		\begin{array}{ll}
			\dfrac{|\Phi|n^2}{\pi^2} \left( G_{{n+1 \over 2}} + \Big( \frac{1}{a}+\frac{1}{c} \Big) H_{{n+1 \over 2}} \right),
			& n \text{ odd},
			\vspace{0.15cm}\\
			\dfrac{|\Phi|n^2}{\pi^2}
			\left( G_{{n +2 \over 2}} + \Big( \frac{1}{a}+\frac{1}{c} \Big) H_{{n+2 \over 2}} - {1 \over 2} U_{{n \over 2}} \right)
			+ \dfrac{2|\Phi|}{\pi^2} \Big( {1 \over a+b+c} - {1 \over a} - {1 \over c} \Big),
			& n \text{ even},
		\end{array}
	\right.
\end{align}
where (see \eqref{eq: 1.32})
\begin{equation} \label{eq:Hn}
	H_n
	= \mathcal{G}_{6}
	= \sum_{j=1}^{n-1} \dfrac{1}{j^2}
	= {\pi^2\over 6} -{1\over n} -{1\over 2n^2} -{1\over 6n^3} + \mathcal{O}\big({1\over n^5}\big)
\end{equation}
and
\begin{align*}
	U_n
	& = \sum_{j = 1}^{n}
	\Big( \dfrac{1}{aj^2+bjn+cn^2} + \dfrac{1}{aj^2- bjn+cn^2} + \dfrac{1}{an^2+ bnj+cj^2} + \dfrac{1}{an^2- bnj+cj^2} \Big)
	\\
	& = {1 \over n} \sum_{j = 1}^{n} {1 \over n} 
	\Big( \dfrac{1}{a({j \over n})^2+b{j \over n}+c} + \dfrac{1}{a({j \over n})^2- b{j \over n}+c}
	+ \dfrac{1}{a+ b{j \over n}+c({j \over n})^2} + \dfrac{1}{a- b{j \over n}+c({j \over n})^2} \Big).
\end{align*}
We apply the Euler-Maclaurin summation formula \eqref{eq: 1.9} with $p=4$ to $V_n = n U_n$ for which the associated function
\[
	g(x) = \dfrac{1}{ax^2+bx+c} + \dfrac{1}{ax^2- bx+c} + \dfrac{1}{a+ bx+cx^2} + \dfrac{1}{a- bx+cx^2}
\] 
has $\displaystyle g(0) = 2 \big( {1 \over a} + {1 \over c} \big)$,
$g'(0) = 0$, $\displaystyle g(1) = -g'(1) = 2 \big( {1 \over a-b+c} + {1 \over a+b+c} \big)$, and
\begin{align*}
	\int_0^1 g(x) \, dx
	& = {2 \over \sqrt{|d|}}
	\big( \arctan{2a+b\over\sqrt{|d|}}+\arctan{2a-b\over\sqrt{|d|}}+\arctan{2c+b\over\sqrt{|d|}}+\arctan{2c-b\over\sqrt{|d|}} \big)
	\\
	& = {2 \pi \over \sqrt{|d|}},
\end{align*}
so that we have
\begin{align} 
	U_n
	= {2 \pi \over \sqrt{|d|}} {1 \over n}
	& + \big( {1 \over a-b+c} + {1 \over a+b+c} - {1 \over a} - {1 \over c} \big) {1 \over n^2}
	\nonumber
	\\
	& - {1 \over 6} \big( {1 \over a-b+c} + {1 \over a+b+c} \big) {1 \over n^3}
	+ \mathcal{O} \Big({1 \over n^5}\Big).
	\label{eq:Un}
\end{align}
Using \eqref{eq:Gnasymp}, \eqref{eq:Hn}, and \eqref{eq:Un} in \eqref{eq:Fnf1concrete}, we obtain \eqref{eq:Fnf1}.

\section{Examples associated to root systems of semisimple Lie algebras} \label{sec:square}
Here we give an exposition of our methods on the asymptotic analysis of $F_n$  when $\Phi$ comprises of positive roots of two particular cases of rank 2 complex semisimple Lie algebras. Below $\{ \mathbf{e}_1, \mathbf{e}_2 \}$  denotes the standard basis for $\mathbb{R}^2$.

{\it Example 1.} Let $\Phi$ be the set of positive roots of $\mathfrak{sl}_2\times \mathfrak{sl}_2$, 
$\Phi \cup -\Phi$ is the root system of type $A_1\times A_1$.  
We may pick $\Phi = \{ \mathbf{e}_1, \mathbf{e}_2 \}$, then
\[
	F_n = F_n(f )= \sum_{j,k=0 \atop (j,k)\neq (0,0)}^{n-1} \frac{1}{1-\frac{1}{2}(\cos\frac{2\pi j}{n} + \cos\frac{2\pi k}{n})}, 
\]
and we use the first part of Theorem~\ref{thm:main2} to conclude
%\begin{equation} \label{eq:Fnfsq}
%	F_n =\frac{2}{\pi}n^2\log n+\frac{1}{\pi}\Big(5\log 2+\log \pi+2\gamma-4\log\Gamma(\frac{1}{4})\Big) n^2+\mathcal{O}(\log n).
%\end{equation}
\begin{equation} \label{eq:Fnfsq}
	F_n =\frac{2}{\pi}n^2\log n+\frac{1}{\pi}\Big(2\gamma+\log \big( {32 \pi \over \Gamma({1 \over 4})^4} \big) \Big) n^2+\mathcal{O}(\log n).
\end{equation}

To see this, we apply Theorems~\ref{thm:Fnf1} and \ref{thm:Inf1} for the parameters
$|\Phi|=2, a=1,b=0,c=1$ corresponding to the square lattice to find
\begin{equation} \label{eq:Fnf1sq}
	F_n(f_1) = {2 \over \pi} \, n^2 \log n + {2 \over \pi} \left(\gamma-2\log|\eta(i)|-\frac{2G}{\pi}-\log 2\right) n^2 + \mathcal{O}(\log n),
\end{equation}
where $G$ is Catalan's constant, and
\begin{equation} \label{eq:Inf1sq}
	I_n(f_1) = \dfrac{2}{\pi}n^2\log n + \mathcal{O}(1).
\end{equation}
Below we show
\begin{equation} \label{eq:Infsq}
	I_n(f)=\frac{2}{\pi} \, n^2 \log n+\frac{1}{\pi}(3\log 2-2\log \pi+\frac{4G}{\pi}) \, n^2+\mathcal{O}(1).
\end{equation}
Since $\displaystyle \eta(i)=\frac{\Gamma(\frac{1}{4})}{2\pi^{\frac{3}{4}}}$, use of \eqref{eq:Fnf1sq}, \eqref{eq:Inf1sq} and \eqref{eq:Infsq}
in Theorem~\ref{thm:main2} gives \eqref{eq:Fnfsq}. \vspace{0.1cm}

For proving \eqref{eq:Infsq}, we start from \eqref{eq:Inf}. Although the integration domain $D_n$ (cf. \eqref{eq:D_n}) depends on the parity of $n$, $f$ in \eqref{eq:Fn-concrete} is $2\pi$-periodic in both of its arguments, $I_n(f)$ can be rewritten as
\begin{equation} \label{eq:Infalternative}
 	I_n(f)
	= \frac{1}{\Delta_n^2} 
	\iint_{R_{n}}
	f(\mathbf{x}) \, d\mathbf{x}
	\qquad
	\text{with}
	\qquad
	R_n = \left[ -\pi,\pi \right]^2 \backslash \left[ -\frac{\pi}{n},\frac{\pi}{n} \right]^2
\end{equation}
for all $n$. For the square lattice \eqref{eq:Infalternative} takes on the form
\[
	I_n(f) = \frac{n^2}{4\pi^2} \iint_{R_n} \dfrac{1}{1-\frac{1}{2}(\cos x+\cos y)}\,dx \, dy.
\]
Using the symmetry of the region $R_n$ along with the evenness of the integrand with respect to both
$x$ and $y$, we now have
\[
	I_n(f) = \frac{n^2}{\pi^2} \iint_{R^+_n} \dfrac{1}{1-\frac{1}{2}(\cos x+\cos y)}\,dx \, dy
\]
where $\displaystyle R^+_n = [0,\pi]^2 \backslash [0, {\pi \over n}]^2$. The change of variables
$\displaystyle \tan{\frac{1}{2}x}=u\geq 0$ and $\displaystyle \tan{\frac{1}{2}y}=v\geq 0$, in turn,
yields
%(then we have $\cos x=\frac{1-u^2}{1+u^2}$, $\cos y=\frac{1-v^2}{1+v^2}$, $d x=\frac{2}{1+u^2}du$, and $d y=\frac{2}{1+v^2}dv$),
\[
	I_n(f) = \frac{4n^2}{\pi^2} \iint_{P_n} \dfrac{1}{u^2+v^2+2u^2v^2}\,du \, dv
\]
with $P_n = [0,\infty)^2 \backslash [0,b_n]^2$ and $\displaystyle b_n = \tan\big(\frac{\pi}{2n}\big)$.
Passing to the polar coordinates and then using the change of variables $r\sin(2\alpha)=\sqrt{2}t$ in the
inner integral, we obtain
\begin{align}
	I_{n}(f)
	& =  {8 n^2 \over \pi^2} \int_0^{\pi/4} \!\!\!\int_{\frac{b_n}{\cos\alpha}}^\infty \frac{r \, dr \, d\alpha}{r^2+2r^4\cos^2\alpha \sin^2 \alpha}
	= {8 n^2 \over \pi^2} \int_0^{\pi/4} \!\!\!\int_{\sqrt{2} b_n\sin \alpha}^\infty \frac{dt \, d\alpha}{t(1+t^2)}
	\label{eq:pointme}
	\\
	& = {4 n^2 \over \pi^2} \int_0^{\pi/4} \log\Big(\frac{t^2}{1+t^2}\Big) \Big|_{\sqrt{2} b_n\sin \alpha}^{\infty} \, d\alpha
	= {4 n^2 \over \pi^2}\int_0^{\pi/4} \log\Big(\frac{1+2b_n^2\sin^2\alpha}{2b_n^2\sin^2\alpha}\Big)\,d\alpha.
	\nonumber
\end{align}
Since
\[
	\log\Big(\frac{1+2b_n^2\sin^2\alpha}{2b_n^2\sin^2\alpha}\Big)
	=-\log(2b_n^2\sin^2\alpha) + \mathcal{O}\big(\frac{1}{n^2}\big)
\]
and
\[
	\int_0^{\pi/4} \log(\sin\alpha)\,d\alpha=-\frac{\pi}{4} \log 2-\frac{G}{2},
\]
we get (using $\log(\tan x) = \log(x) + \mathcal{O}(x^2)$ as $x \to 0^+$)
\begin{align*}
	\int_0^{\pi/4} \log\Big(\frac{1+2b_n^2\sin^2\alpha}{2b_n^2\sin^2\alpha}\Big)\,d\alpha
	& = -\frac{\pi}{4} \log(2b_n^2)+G+\frac{\pi}{2}\log2 +\mathcal{O}\big(\frac{1}{n^2}\big)
	\\
	& = \frac{\pi}{2}\log n + \big( \frac{3\pi}{4}\log 2-\frac{\pi}{2}\log \pi+G \big) +\mathcal{O}\big(\frac{1}{n^2}\big),
\end{align*}
and \eqref{eq:Infsq} follows.

{\it Example 2.} Let $\Phi$ be the set of positive roots of $\mathfrak{so}_5$, $\Phi \cup - \Phi$ is the root system of type $B_2$.  
We may pick  $\Phi=\{ \mathbf{e}_1, \mathbf{e}_2, \mathbf{e}_1-\mathbf{e}_2,\mathbf{e}_1+\mathbf{e}_2\}$, and choose
$\{ \mathbf{e_1},\mathbf{e_2} \}$ as basis. Then with our notation,
$$
	S =
	\begin{bmatrix}
		1 & 0 & 1 & 1
		\\
		0 & 1 & -1 & 1
	\end{bmatrix}^T, \ S^TS=\begin{bmatrix}
		3 & 0
		\\
		0 & 3
	\end{bmatrix}=\begin{bmatrix}
		a & b/2
		\\
		b/2 & c
	\end{bmatrix},
	\
	\sqrt{\det(S^TS)}=3,
	\ d=-36,
$$
$$\phi(\mathbf{x})=\frac{1}{8}(e^{\pm i x_1}+e^{\pm i x_2}+e^{\pm i (x_1-x_2)}+e^{\pm i (x_1+x_2)}),$$
$$\psi(\mathbf{x})=1-\frac{1}{4}(\cos x_1+\cos x_2+\cos (x_1-x_2)+\cos (x_1+x_2)), $$
and
\[
	F_n (f) = \sum_{j,k=0 \atop (j,k)\neq (0,0)}^{n-1}
	\frac{1}{1-\frac{1}{4}(\cos \frac{2\pi j}{n}+\cos \frac{2\pi k}{n}+\cos\frac{2\pi (j-k)}{n}+\cos\frac{2\pi (j+k)}{n})}.
\]
Applying Theorems~\ref{thm:Fnf1} and \ref{thm:Inf1} with $a=c=3$, $b=0$, and $d = -36$, we find
\[
	F_n(f_1)
	= {4 \over 3 \pi} n^2 \log n 
	+ {4 \over 3\pi} 
	\Big(\gamma - \log 2  - 2 \log |\eta(i)| - \frac{2G}{\pi} \Big) n^2 + \mathcal{O}(1)
\]
and
\[
	I_n(f_1)=\frac{4}{3\pi}n^2\log n + \mathcal{O}(1).
\]
The same procedure between \eqref{eq:Infalternative} and \eqref{eq:pointme} yields
\begin{align*}
	I_{n}(f)
	& = \frac{n^2}{4\pi^2} \iint_{R_n} \dfrac{dx \, dy}{1-\frac{1}{4}(\cos x + \cos y + \cos (x-y) + \cos (x+y))}
	\\
	& = \frac{n^2}{\pi^2} \iint_{R_n^+} \dfrac{dx \, dy}{1-\frac{1}{4}(\cos x + \cos y + \cos (x-y) + \cos (x+y))}
	\\
	& = \frac{8n^2}{\pi^2} \iint_{P_n} \dfrac{du \, dv}{3u^2+3v^2+2u^2v^2}
	= \frac{16n^2}{\pi^2} \int_{0}^{\pi/4} \!\!\! \int_{{b_n \over \cos \alpha}}^{\infty} \dfrac{r\, dr\, d\alpha}{3r^2+2r^4\, \cos^2\alpha \, \sin^2 \alpha}
	\\
	& = \frac{16n^2}{\pi^2} \int_{0}^{\pi/4} \!\!\! \int_{\sqrt{2} b_n \sin \alpha}^{\infty} \dfrac{dt \, d\alpha}{t(3+t^2)}
	= \frac{8n^2}{3\pi^2} \int_{0}^{\pi/4} \log\left(\frac{3+2b_n^2\sin^2\alpha}{2b_n^2\sin^2\alpha}\right) \, d\alpha  
\end{align*}
so that
$$I_n(f)
%=\displaystyle \left(\frac{n}{\pi}\right)^2 \frac{8}{3}\left(\log 3\frac{\pi}{4}+
%\frac{3\pi}{4}\log 2+G-\frac{\pi}{2}\log \pi+\frac{\pi}{2}\log n\right)\right)+O(1)\\
=\displaystyle \frac{4}{3\pi}n^2\log n+\frac{2n^2}{3\pi}\left(\log\left(\frac{24}{\pi^2}\right)+\frac{4G}{\pi}\right)+\mathcal{O}(1).$$
	
Using Theorem 2,
%\begin{align*}
%F_n(f)&=I_n(f)-I_n(f_1)+F_n(f_1)+\mathcal{O}(\log n)\\
%%&=\frac{4}{3\pi}n^2\log n+\frac{n^2}{\pi}\left(\frac{2}{3}\log 3+\frac{4}{3}\gamma+\frac{2}{3}\log 2-\frac{4}{3}\log \pi-\frac{8}{3}\log|\eta(i)|\right)+O(\log n)\\
%&=\frac{4}{3\pi}n^2\log n+\frac{n^2}{3\pi}\left(4\gamma+2\log 3+10\log 2+2\log \pi-8\log\Gamma(\frac{1}{4})\right)+\mathcal{O}(\log n).
%\end{align*}

\begin{align*}
	F_n(f)
	& = I_n(f)-I_n(f_1)+F_n(f_1)+\mathcal{O}(\log n)
	\\
	& = \frac{4}{3\pi}n^2\log n
	+ \frac{2}{3\pi} \Big(2\gamma + \log \big( {96 \pi \over \Gamma({1 \over 4})^4}\big) \Big) n^2
	+ \mathcal{O}(\log n).
\end{align*}

\section{Further Problems} \label{sec:fp}

In view of Theorem 2,  one may consider carrying out the explicit calculations taking $m=2$ (and therefore obtaining results with more precision)
for some particular lattices. It is also desirable to obtain the exact rational values for $F_n$ at least in some special cases (e.g. in the cases of the triangular
lattice and the square lattice for some sequences of $n$  which may be the sequence of primes or the sequence of powers of $2$). This would involve
employing the tools of algebraic number theory.

\section*{Acknowledgment} This research is supported by Bo\u{g}azi\c{c}i University Research Fund Grant Number 13387.

\end{document}